\documentclass[11pt]{article}

\usepackage{amssymb,amsmath,amsthm,latexsym}

\usepackage{graphicx}
\usepackage{overpic}
\usepackage{tikz}

\usepackage{comment}

\DeclareMathOperator{\dist}{dist}

\DeclareMathOperator{\Bal}{Bal}
\DeclareMathOperator{\supp}{supp}

\DeclareMathOperator{\BIS}{\textbf{BIS}}
\DeclareMathOperator{\FIS}{\textbf{FIS}}
\DeclareMathOperator{\TIS}{\textbf{TIS}}
\DeclareMathOperator{\UIS}{\textbf{UIS}}

\let\Im\undefined

\DeclareMathOperator{\Im}{Im}

\newcommand{\ds}{\displaystyle}

\newtheorem{theorem}{Theorem}[section]
\newtheorem{lemma}[theorem]{Lemma}
\newtheorem{proposition}[theorem]{Proposition}

\theoremstyle{definition}
\newtheorem{definition}[theorem]{Definition}
\newtheorem{remark}[theorem]{Remark}

\numberwithin{equation}{section}

\title{A vector equilibrium problem for symmetrically
	located point charges on a sphere}
\author{Juan G. Criado del Rey 
	\footnote{Department of Mathematics, Katholieke Universiteit Leuven, Belgium, Email: juan.gcriadodelrey@kuleuven.be. 
		Supported by FWO Flanders project EOS 30889451.}	 
		 \and  Arno B.J. Kuijlaars
\footnote{Department of Mathematics, Katholieke Universiteit Leuven, Belgium, Email: arno.kuijlaars@kuleuven.be. Supported by long term structural funding-Methusalem grant of the Flemish
	Government, and by FWO Flanders projects
	 EOS 30889451, G.0864.16 and 
	 G.0910.20.}	  
}

\date{}

\begin{document}

\maketitle

\begin{abstract}
	We study the equilibrium measure on the  two dimensional
	sphere in the presence of an external field generated
	by $r+1$ equal point charges that are symmetrically
	located around the north pole. The support of 
	the equilibrium measure is known as the droplet.
	The droplet has a motherbody which we 
	characterize by means of a vector equilibrium
	problem (VEP) for $r$ measures in the complex plane. 
	
	The model undergoes two transitions which is
	reflected in the support of the first component
	of the minimizer of the VEP, namely the support
	can be a finite interval containing $0$, the
	union of two intervals, or the full half-line.
	The two interval case corresponds to 
	a droplet with two disjoint components, and
	it is analyzed by means of a genus one Riemann surface.
	\end{abstract}

\section{Introduction} \label{sec1}

\subsection{Equilibrium on the sphere} \label{subsec11}

This paper deals with an electrostatic equilibrium problem 
for free charges on the unit sphere $\mathbb S^2 \subset
\mathbb R^3$ with logarithmic interaction
under the influence of a finite number of fixed point charges 
\cite{BrDrSaWo18, CrKu19+, Dr02,  LeDr19+}.
Suppose there are $r+1$ fixed charges at points $p_0, \ldots, p_r$ on $\mathbb S^2$,
and each $p_j$ carries a charge $a_j > 0$, leading to
a charge distribution
\begin{equation} \label{eq:sigma} 
\sigma = \sum_{j=0}^r a_j \delta_{p_j}. \end{equation}
Then there exists an equilibrium measure $\mu_{\sigma}$ in 
the presence of the fixed charges 
that is the unique probability measure on $\mathbb S^2$
that satisfies for some constant $\ell$,
\begin{equation} \label{eq:musigma1}
\begin{aligned} 
	U^{\mu_{\sigma}} + U^{\sigma}
	& = \ell, \quad \text{ on } D_{\sigma} = \supp(\mu_{\sigma}), \\
	U^{\mu_{\sigma}} + U^{\sigma}
	& \geq \ell, \quad  \text{ on } \mathbb S^2,
 \end{aligned} \end{equation}
where we use
\[ U^{\mu}(x) = \int \log \frac{1}{\|x-y\|} d\mu(y) \]
to denote the logarithmic potential of a measure $\mu$.
The domain $D_{\sigma}$ is known as the droplet, and it determines the measure $\mu_{\sigma}$ since 
\begin{equation} \label{eq:musigma} 
	\mu_{\sigma} =  (\lambda(D_{\sigma}))^{-1} \lambda_{D_{\sigma}} 
	\end{equation}
where $\lambda_D$ denotes the restriction to $D$ of the 
normalized Lebesgue measure $\lambda$ on the sphere. 
It is known that
\begin{equation} \label{eq:lambdaD}
	\lambda(D_{\sigma}) = \frac{1}{1+\sigma(\mathbb S^2)} 
	= \frac{1}{1+ \sum_{j=0}^r a_j}. \end{equation}
	see e.g.\ \cite[Appendix A]{CrKu19+}.

A motherbody (or a potential theoretic skeleton \cite{GuTeVa14})
for $D_{\sigma}$ is a probability measure $\sigma^*$
supported on a one-dimensional subset of $\mathbb S^2$ 
(i.e., a curve, or a system of curves) such that for  some constant $\ell^*$,
\begin{equation} \label{eq:musigma2}
\begin{aligned}
U^{\sigma^*} & =  U^{\mu_{\sigma}} + \ell^*,  \quad
\text{ on } \mathbb S^2 \setminus D_{\sigma}, \\
U^{\sigma^*} & \geq U^{\mu_{\sigma}} + \ell^*, \quad \text{ on } 
\mathbb S^2. 
\end{aligned}
\end{equation}  
Motherbodies are connected to a variety of topics
in applied complex analysis, such as quadrature domains
and Schwarz functions
\cite{AhSh76, Cr05, GuSh05, LeMa16}, 
partial balayage and Hele-Shaw flows \cite{Gu02},
orthogonal 
polynomials in the complex plane \cite{BaBeLeMc15, BlKu12, LeWa17}
and  normal matrix models \cite{BlSi20, TeBeAgZaWi05}.

The aim of this paper is to construct such a motherbody 
by means of a vector equilibrium problem in the
special situation where the points are
in a symmetric position around a distinguished point
on the unit sphere, that without loss of generality
we can take as the north pole. More
precisely, we assume that the distance to
the north pole is the same for each point $p_j$, 
which means that the points are on a circle of constant latitude.
On this circle the points are evenly distributed, like
vertices of a regular $r+1$-gon.
We also assume 
\[ a_j = a, \qquad \text{for } j=0,\ldots, r. \]
In this situation we are able to compute
the motherbody, which, because of rotational symmetry,
is supported on $r+1$ meridians 
(lines of constant longitude) that connect
the north and south poles. From the motherbody we go 
on to construct the droplet $D_{\sigma}$.

With fixed points $p_0, \ldots, p_r$, 
the droplet and the support of the motherbody
decrease as we increase $a$.
We find three possible situations and the transitions 
between them. 
\begin{itemize}
	\item For small $a > 0$, the droplet is big and
	the complement $\mathbb S^2 \setminus D_{\sigma}$ consists
	of $r+1$ disjoint spherical caps, one around each
	of the points $p_j$. The motherbody is supported
	on the full meridians with a positive
	density. 
	
	\item For a first critical value $a_{1,cr}$, the
	spherical caps are tangent to each other. 
	 The density of the motherbody
	becomes zero at the points of tangency.
	
	\item
	For $a > a_{1,cr}$ the droplet is no longer 
	the complement
	of disjoint spherical caps.
	 For $a$ somewhat 
	larger than $a_{1,cr}$ the droplet will have
	two connected components (provided $r \geq 2$), one containing
	the north pole and the other one the south pole.
	The motherbody is not fully supported anymore.
	On each meridian the  support has two parts, one
	with the north pole and one with the south pole. 
	
	\item For a second critical value $a_{2,cr}$ one of
	the components disappers. If the points $p_j$ are in
	the northern hemisphere, 
	then the component containing the north pole disappears.
	Also the parts of the motherbody containing the
	north pole have disappeared at the second critical value.
	 
	\item For larger $a >  a_{2,cr}$
	the droplet $D_{\sigma}$ is simply connected containing  
	the south pole
	(assuming again that the points $p_j$
	are in the northern hemisphere). The support of the 
	motherbody consists of $r+1$ segments containing
	the south pole, one segment along each meridian.
	
	\item As $a \to \infty$, the droplet and the support
	of the motherbody further shrink to
	the south pole. 
	 \end{itemize}

\subsection{The case $r=1$}

\begin{figure}[t]
	\centering
	\includegraphics[width=0.52\textwidth]{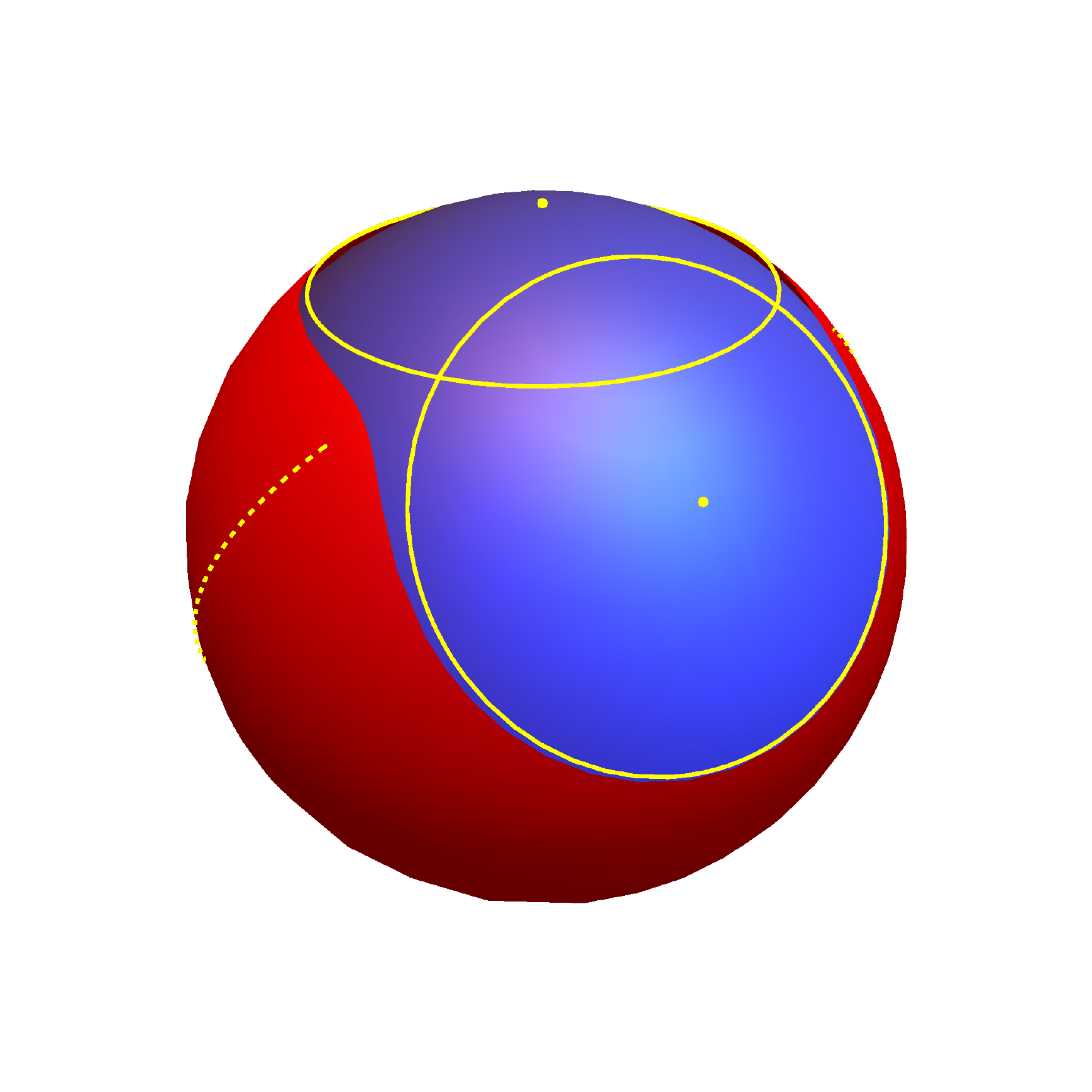}\hspace{-1cm}
	\includegraphics[width=0.52\textwidth]{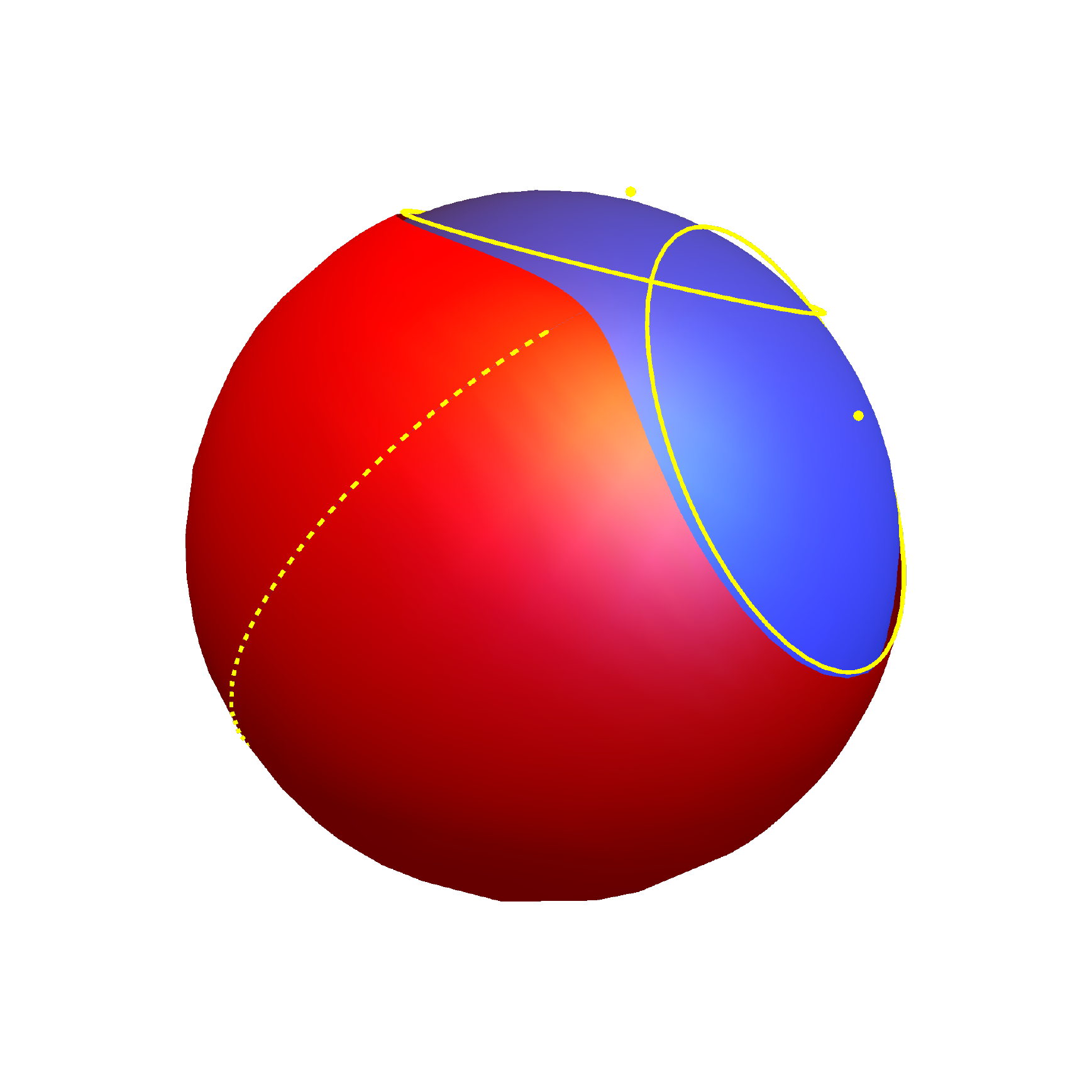}
	\caption{Picture of the droplet (red region)
	in case $r=1$ and $a > a_{1,cr}$. The 
		spherical caps centered at $p_0$ and $p_1$ with geodesic radii $a/(1+2a)$ are also represented, as well as the support of the motherbody (dashed line inside the droplet). The boundary of the droplet is mapped
		by stereographic projection onto an
		ellipse in the complex plane \cite{CrKu19+}. \label{figure1}}
\end{figure}

For $r=1$ the two spherical caps are tangent at the north pole 
at the critical value $a_{1,cr}$. Then there is no second
critical $a$-value since for each $a > a_{1,cr}$ the droplet
is simply connected. The support of the motherbody is 
an interval along the big circle that separates the
two points $p_0$ and $p_1$. 
See Figure~\ref{figure1} that is taken from \cite{CrKu19+} 
and compare also with  \cite[Figure~4]{BrDrSaWo18}.
This situation was analyzed
in \cite{CrKu19+} and it was shown that the boundary
of the droplet is mapped by stereographic projection to an ellipse in the complex plane.
This fact can  also be deduced from earlier work 
by Gustafsson and Tkachev in \cite[Example 3]{GuTk11}.

The approach of \cite{CrKu19+} is to first characterize
the motherbody by means of an equilibrium problem from
logarithmic potential theory \cite{Ra95, SaTo97}. This
equilibrium problems asks for the minimizer of
\begin{equation} \label{eq:VEPforris1} 
\iint \log \frac{1}{|x-y|}d\mu(x)d\mu(y) 
+ 2\int V(x) d\mu(x) \end{equation}
among probability measures $\mu$ on $\mathbb R$, with
\begin{equation} \label{eq:Vforris1} V(x) = \frac{1+a}{2} \log(x^2+b^{-2}) - \frac{a}{2} \log(x^2+b^2),
\end{equation}
where $\pm ib$, $b > 1$, are the images of the two
points $p_0$, $p_1$ under stereographic projection
onto the complex plane. The minimizer
is calculated explicitly in \cite[Theorem 1.6]{CrKu19+},
see also \cite{OrSaWi19}. The external field is only
weakly admissible \cite{HaKu12, Si05} and for a fixed
$b >1$ there is a critical value $a_{cr}$ such that
the minimizer $\mu_V$ is compactly supported if and only if
$a > a_{cr}$. Out of the Stieltjes transform of $\mu_V$
a meromorphic function $S$ is then constructed  that is shown
to be the spherical Schwarz function of a certain domain
$\Omega$ in the sense that its boundary
is characterized by
\[ \partial \Omega:  \quad S(z) = \frac{\bar{z}}{1+|z|^2}. \]
After pulling back to the sphere with inverse 
stereographic projection, the domain $\Omega$
is then proved to give the droplet $D_{\sigma}$ and $\mu_V$ gives the motherbody. 

\subsection{Stereographic projection and removal of
symmetry}

In this paper we extend the approach of \cite{CrKu19+} 
to $r+1$ points on the sphere. 
As in \cite{CrKu19+} we project onto the complex plane where we do all calculations. Instead of
the equilibrium problem \eqref{eq:VEPforris1}, \eqref{eq:Vforris1} we study 
a vector equilibrium problem for a vector of $r$
measures. This will be described in section \ref{subsec21} below.
In this section we first describe what we aim
to achieve in the complex plane.

We move from the sphere to the complex
plane by stereographic projection, where the south pole
is mapped to $0$ and the north pole to $\infty$. The points
$p_0, \ldots, p_r$ are projected to $r+1$ points with absolute
value $q^{-\frac{1}{r+1}}$ for some number $q > 0$. 
The projected points will be the solutions of the equation $z^{r+1} + q = 0$,
namely  
\begin{equation} \label{eq:qandpj} 
	p_j \mapsto   q^{-\frac{1}{r+1}}  e^{i \theta_j},
	\qquad \theta_j = \frac{\pi}{r+1} + \frac{2 j \pi}{r+1},
	\qquad \text{for } j = 0,1, \ldots, r.
	\end{equation}
The case $q < 1$ corresponds to points $p_j$ in 
the northern hemisphere, and $q > 1$ to points in 
the southern hemisphere.

The angles $\theta_j$ are chosen in such a way that the meridians separating
the points $p_0, \ldots, p_r$ at equal distances are mapped to
the $r+1$-star
\[ \{ z \in \mathbb C \mid z^{r+1} \in [0,\infty) \}. \]
The droplet $D_{\sigma}$ is mapped to a domain $\Omega \subset \mathbb C \cup \{\infty\}$, 
and $\mu_{\sigma}$ is mapped to its pushforward on $\Omega$
which takes the form 
\begin{equation} \label{eq:dmuOmega} d \mu_{\Omega}(z)   = \left.  \frac{dA(z)}{t \pi (1+|z|^2)^2} \right|_{\Omega} 
\end{equation}
where  $dA(z)$ is the planar Lebesgue measure on $\mathbb C$.
and
 \begin{equation} \label{eq:tina} 
t = \frac{1}{1+\sigma(\mathbb S^2)} 
	= \frac{1}{1+(r+1)a}  \end{equation}

The properties \eqref{eq:musigma1} translate into
\begin{equation} \label{eq:muOmega1} 
	U^{\mu_{\Omega}}(z)  + a \log \frac{1}{|z^{r+1}+q^{-1}|} +
	\frac{1+(r+1)a}{2}  \log \left(1+ |z|^2\right)    
	\begin{cases} = c_1, & \quad  z \in \Omega, \\
	\geq c_1, & \quad z \in \mathbb C. 
	\end{cases} \end{equation}
for some constant $c_1$.

The motherbody $\sigma^*$ (that we are looking
for in this paper and whose existence we do not a priori
assume) satisfying
\eqref{eq:musigma2} corresponds to a probability
measure $\mu^*$ on $\{z \mid z^{r+1} \in [0,\infty) \}$
with the property that
\begin{equation} \label{eq:muOmega2}
\begin{aligned} 
	U^{\mu^*} & = U^{\mu_{\Omega}} + c_2, \quad 
	\text{ on } \mathbb C \setminus \Omega, \\
	U^{\mu^*} & \geq U^{\mu_{\Omega}} + c_2, \quad	
	\text{ on } \mathbb C,
	\end{aligned}
\end{equation} 
for some other constant $c_2$.
The aim of the paper is to construct the domain $\Omega$
and measures
$\mu^*$ and $\mu_{\Omega}$ satisfying the conditions
\eqref{eq:muOmega1} and \eqref{eq:muOmega2}.

The probability measures $\mu^*$ and $\mu_{\Omega}$ 
will be invariant under rotations around the origin 
over angle $\frac{2\pi}{r+1}$. For our computations
it will be convenient to remove the rotational symmetry, and 
change variables $z \mapsto z^{r+1}$.
Then $\mu^*$ will correspond to a probability measure $\mu_1$
on $[0,\infty)$, and $\mu_{\Omega}$ to a probability measure $\mu_U$ on the set
\begin{equation} \label{eq:Udef} 
	U = \{ z^{r+1} \mid z \in \Omega \} \end{equation}
and $\mu_U$ takes the form 
\begin{equation} \label{eq:dmuU} 
	d \mu_U(z) = \left. \frac{1}{(r+1)t \pi} \frac{dA(z)}{
|z|^{\frac{2r}{r+1}} \left(1+ |z|^\frac{2}{r+1}\right)^2} 
	\right|_{U}. 
	\end{equation}
which comes from applying the change of variables to
\eqref{eq:dmuOmega}
		
Our approach will be to construct $\mu_1$ first as
the first component of the minimizer of a vector
equilibrium problem (VEP) for $r$ measures. Besides $\mu_1$
there will be further measures $\mu_2,\ldots, \mu_{r}$
that play auxiliary roles. They do not have a direct
interpretation for the problem at hand, though.

In the next section we will state the VEP without trying to motivate the form
that it takes. It is actually by no means obvious that this
VEP is relevant for our problem, and it will be 
our main result that $\mu_1$ after
symmetrization gives indeed a measure $\mu^*$ that can be
identified as the
image of the motherbody under stereographic projection.
However, for $r=1$, the VEP is an equilibrium
problem for one measure that, after symmetrization,
can be identified with \eqref{eq:Vforris1}. 

The VEP gives rise to an algebraic structure
and this will allow us to find a domain
$U$ with a measure \eqref{eq:dmuU}. Through \eqref{eq:Udef}
we find a domain $\Omega$ with rotational symmetry and the
measure $\mu_{\Omega}$ as in \eqref{eq:dmuOmega}. 
We prove that it has the properties \eqref{eq:muOmega1}
and \eqref{eq:muOmega2}.

The VEP depends on two parameters
$q > 0$ and $t \in (0,1)$, that ultimately will
play the roles of the parameters appearing in \eqref{eq:qandpj} and \eqref{eq:tina}, as we will show
in the end.

\section{Statement of results} \label{sec2}
\subsection{Vector equilibrium problem} \label{subsec21}

Let $r \geq 2$ be an integer, and let $q > 0$, $0 < t < 1$
be real parameters.
Our starting point is a vector equilibrium problem 
that asks to  minimize the energy functional
\begin{multline} \label{eq:VEPF}
\mathcal E(\mu_1, \mu_2, \ldots, \mu_r) =  
\sum_{j=1}^{r} I(\mu_j) - \sum_{j=1}^{r-1}
	I(\mu_j,\mu_{j+1}) \\
	+ \frac{1-t}{t} I\left(\mu_1, \delta_{-q^{-1}}\right)
		- \frac{r+t}{t} I\left(\mu_r, \delta_{(-1)^r q}\right), 
		\end{multline}
depending on $r$ measures. Here $\delta_{-q^{-1}}$ and $\delta_{(-1)^r q}$ denote Dirac point masses. As usual we write
\[ I(\mu,\nu) =
	\int U^{\mu} d\nu = \iint \log \frac{1}{|x-y|} d\mu(x) d\nu(y) \] 
for the mutual logarithmic energy of $\mu$ an $\nu$,
and $I(\mu) = I(\mu,\mu)$ for the logarithmic energy of $\mu$.

Our aim is to minimize \eqref{eq:VEPF} over a
vector of measures satisfying certain conditions. We emphasize that a measure (without any adjective) 
will always refer to a positive measure. We also encounter
negative measures or signed measures in this paper, 
but in such a context the adjective will always be mentioned.

\begin{definition} \label{def21}
	The vector equilibrium problem (VEP)
	asks to minimize the energy functional \eqref{eq:VEPF} over 
vectors $(\mu_1, \ldots, \mu_r)$  of measures subject to the conditions
\begin{itemize} 
	\item[(a)] $\supp(\mu_j) \subset \Delta_j$ for every $j$,
	where  
	\begin{equation} \label{eq:Deltaj}
	\Delta_j = \begin{cases} [0,\infty), & \text{ if $j$ is odd},  \\
	(-\infty,0], & \text{ if $j$ is even},
	\end{cases}
	\end{equation}
	\item[(b)] the total mass of $\mu_j$ is 
	\begin{equation} \label{eq:massmuj}
	\mu_j(\Delta_j) = 1 + \frac{j-1}{t}, \quad
	\text{ for }  j =1, \ldots, r. 
	\end{equation}
\end{itemize}
\end{definition}

Throughout the paper we will write
\begin{equation}  \label{eq:mu0}
	\mu_0 = \left(1- \frac{1}{t} \right) \delta_{-q^{-1}}, 
	\qquad \mu_{r+1} = \left(1 + \frac{r}{t}\right) \delta_{(-1)^r q}.  
\end{equation} 
Then \eqref{eq:massmuj} is also satisfied for 
$j \in \{0, r+1\}$, but note that $\mu_0$ is
a negative measure (since $0 < t < 1$). Moreover,  \eqref{eq:VEPF} takes the compact form
\begin{equation} \label{eq:VEPF2} 
\mathcal E(\mu_1, \ldots, \mu_r) = 	\sum_{j=1}^{r} I(\mu_j) - \sum_{j=0}^{r}
	I(\mu_j,\mu_{j+1}), \end{equation}
that includes $\mu_0$ and $\mu_{r+1}$ as well, but $\mu_0$
and $\mu_{r+1}$  remain fixed in the VEP.

Vector equilibrium problems were first introduced
by Gonchar and Rakh\-manov
in their study of Hermite-Pad\'e approximation \cite{GoRa81,GoRa85}, see also \cite{NiSo91}.
They also appear in ensembles of random matrices that
are related to multiple orthogonal polynomials, see
\cite{ApKu11, Ku10} and references cited therein.

The  energy functional \eqref{eq:VEPF2} involves an attraction 
between neighboring measures that is of Nikishin type,
and this has appeared in a number of situations before.
What is special is that the total masses \eqref{eq:massmuj}
are in an arithmetic progression that is
increasing with steps $1/t$. 
It is more common that the masses are
in an arithmetic progression that decreases from $1$ to $0$ 
 see e.g.\ \cite{DuKu08} and the examples
 in \cite{ApKu11, Ku10}.

The VEP of Definition \ref{def21} is weakly admissible 
in the sense of \cite{HaKu12} as we show next. 

\begin{lemma} \label{lemma22}
	The vector equilibrium problem is weakly admissible.
	There is a unique minimizer, denoted $(\mu_1, \ldots, \mu_r)$. 
	The measures $\mu_2, \ldots, \mu_r$
	have full supports
\begin{equation} \label{eq:supp} 
	\supp(\mu_j) = \Delta_j = (-1)^{j-1} [0,\infty), 
	\quad \text{ for } j=2, \ldots, r. 
\end{equation} 
\end{lemma}
\begin{proof}
	To check the conditions in Assumption 2.1 of \cite{HaKu12}, we write the energy functional \eqref{eq:VEPF2} in the form 
	\[ \sum_{1 \leq i,j\leq r} c_{ij} I(\mu_i,\mu_j)
		+ \sum_{j=1}^r \int V_j d\mu_j \]
	with 
	\begin{equation} \label{eq:lem22proof1} 
	c_{ij} = \begin{cases}
		   1 & \text{ if } i = j, \\
		   -\frac{1}{2} & \text{ if } |i-j| = 1, \\
		   0 & \text{ otherwise}, \end{cases} 
		   \end{equation}
	and
	\begin{equation} \label{eq:lem22proof2} 
		V_j(x) = \begin{cases} 
		- U^{\mu_0}(x) = (1-\frac{1}{t}) \log|x+q^{-1}|,
			& \text{ if } j = 1, \\
		- U^{\mu_{r+1}}(x) = (1 + \frac{r}{t}) 
		\log | x - (-1)^r q|, & \text{ if } j = r, \\
		\equiv 0, & \text{ otherwise}. 
		\end{cases}  \end{equation}
	The interaction matrix $C = (c_{ij})$ is symmetric and positive definite, and each $V_j$ is continuous on 
	$\Delta_j$, since $-q^{-1} \not\in \Delta_1$
	and $(-1)^r q \not\in \Delta_r$.
	
	The prescribed total masses $m_j = \mu_j(\Delta_j)$ from \eqref{eq:massmuj} come in an arithmetic progression which implies by \eqref{eq:lem22proof1} that 
	\begin{equation} \label{eq:lem22proof3} 
	\sum_{j=1}^r c_{ij} m_j = 0,
		\qquad \text{ for }  i=2, \ldots, r-1, 
	\end{equation}
	and also
	\begin{equation} \label{eq:lem22proof4} 
		\sum_{j=1}^r c_{ij}  m_j 
		= \begin{cases}  \frac{1}{2} m_0 
		= \frac{1}{2} (1-\frac{1}{t}),
			& \text{ for } i =1, \\
		\frac{1}{2} m_{r+1}
		= \frac{1}{2} (1+ \frac{r}{t}),
		& \text{ for } i = r. 
		\end{cases} \end{equation}
	It follows from \eqref{eq:lem22proof2}, \eqref{eq:lem22proof3} and \eqref{eq:lem22proof4}
	that, for every $i=1, \ldots, r$,
	\[  V_i(x) - \left(\sum_{j=1}^r c_{i,j} \mu_j(\Delta_j) \right)
			\log(1+|x|^2) \to 0, 
		\quad	\text{ as } x \in \Delta_j \to \pm \infty. \]
	Thus all conditions of Assumption 2.1 in \cite{HaKu12}
	are satisfied, and the VEP is weakly admissible.
	Then there is 
	a unique minimizer by \cite[Corollary 2.7]{HaKu12}. 
	
	Given the other measures, the problem for
	$\mu_j$ (for $1 \leq j \leq r$) is to minimize
	\[ I(\mu_j) - I(\mu_j, \mu_{j-1}+ \mu_{j+1}) \]
	among measures on $\Delta_j$ with total mass
	\eqref{eq:massmuj}.
	Since $\mu_{j-1} + \mu_{j+1}$ is a positive measure
	for $j \geq 2$, it follows that
	 $\mu_j$ is a balayage measure (see \cite{SaTo97} for
	 the notion of balayage)
	\begin{equation} \label{eq:balmuj} 
	\mu_j = \frac{1}{2} \Bal \left(\mu_{j-1}+\mu_{j+1}, \Delta_j \right) 
	\qquad \text{ for } j =2, \ldots, r, \end{equation}
	and $\mu_j$ has full support for $j \geq 2$,
	see also \eqref{eq:theo2proof3} for the expression
	of the density of the balayage of a  measure
	on $(-\infty,0]$ onto $[0,\infty)$. There is a similar
	formula for the balayage of a measure on $[0,\infty)$
	to $(-\infty,0]$ that shows that it has indeed  a full support.
\end{proof}

The balayage property \eqref{eq:balmuj} means that 
\begin{equation} \label{eq:Umuj} 
	2 U^{\mu_j} = U^{\mu_{j-1}} + U^{\mu_{j+1}} 
	\quad \text{ on } \Delta_j, \qquad \text{ for } j=2,
	\ldots, r, \end{equation}
and this will be important for us in what follows.

The measure $\mu_1$ is the main player in the game. The argument
in the proof of Lemma \ref{lemma22} leading to \eqref{eq:balmuj}
does not  work for $j=1$, since $\mu_0$
is a negative measure. Therefore the balayage of
$\mu_0 + \mu_2$ onto $\Delta_1 = [0,\infty)$ is not
necessarily positive on the full half-line. 
However, if it is positive then \eqref{eq:balmuj} and \eqref{eq:Umuj}
hold for $j=1$ as well, and then also $\mu_1$ has a 
full support. It turns out
that this happens for $t$ sufficiently large (i.e.,
sufficiently close to $1$).

Our first main result is about the structure of the
support $\Sigma_1 = \supp(\mu_1)$ of $\mu_1$.
There are four possible cases that will be indicated 
with acronyms 
$\BIS$ = Bounded Interval Support, 
$\UIS$ = Unbounded Interval Support,
$\TIS$ = Two Interval Support, and 
$\FIS$ = Full Interval Support. 

In situations where we want to emphasize the 
dependence on $t$ of the various notions that
we introduced (and of others that are still to come), we append
a subscript $t$. Hence we write for example $\mu_{1,t}$,
$\Sigma_{1,t}$,  and so on.
	
\begin{theorem} \label{theorem23}
	Fix $ q > 0$. Let $(\mu_1, \ldots, \mu_r)$ be
	the minimizer of the vector equilibrium problem
	depending on the parameter $t \in (0,1)$.
\begin{enumerate} 
	\item[\rm (a)] There are four possible
	cases for $\Sigma_1 = \supp(\mu_1)$, depending on $t$, namely there exist $0 < x_1 < x_2 < \infty$ such that either
	$\BIS: \Sigma_ 1 = [0,x_1]$, or
	$\UIS: \Sigma_1 = [x_2,\infty)$, or
	$\TIS: \Sigma_1 = [0,x_1] \cup [x_2, \infty)$, or
	$\FIS: \Sigma_1 = [0,\infty)$.
	\item[\rm (b)]
	For each $j = 1, \ldots, r$ the measure
	$t \mu_{j,t}$ increases as a function of $t \in (0,1)$.
	\item[\rm (c)] Suppose $0 < q < 1$. Then $0$ is always
	in the  support of $\mu_1$ (and so $\UIS$ case does not occur for any $t \in (0,1)$).
	\item[\rm (d)] The measure $\mu_1$ has a density
	that is real analytic on the interior of its support
	with 
	a square-root vanishing at $x_1$ in the $\BIS$ and
	$\TIS$ cases, and at $x_2$ in the $\UIS$ and $\TIS$
	cases.  
	
	\item[\rm (e)] There exist constants $c_0 > 0$ and $c_{\infty} >0$ such that
	\begin{equation} \label{eq:mu1at0} 
		\frac{d\mu_1(x)}{dx} = c_0 x^{- \frac{r}{r+1}}
		\left(1+ O\left(x^{\frac{1}{r+1}}\right)\right) \quad \text{ as } x \to 0+ \end{equation}
	in $\BIS$, $\TIS$ and $\FIS$ cases, and
	\begin{equation} \label{eq:mu1atinf} 
	\frac{d\mu_1(x)}{dx} = c_{\infty} x^{- \frac{r+2}{r+1}}
	\left(1+ O\left(x^{-\frac{1}{r+1}}\right)\right) \quad \text{ as } x \to \infty \end{equation}
	in $\UIS$, $\TIS$ and $\FIS$ cases. 
\end{enumerate}
\end{theorem}
The proof of Theorem \ref{theorem23} is in section \ref{sec3}, except for the proof of part (e) 
which is in section \ref{subsec412}.

It follows from Theorem \ref{theorem23} that for $0 < q < 1$, there are two critical values 
\begin{equation} \label{eq:criticalt}
	0 \leq t_{1,cr} \leq t_{2,cr} \leq 1,
	\end{equation} 
depending on $q$,  such that 
we are in $\BIS$ case for $0 < t \leq t_{1,cr}$,
in the $\TIS$ case for $t_{1,cr}< t < t_{2,cr}$, and in
$\FIS$ case for $t_{2,cr} \leq t < 1$. 
For $r \geq 2$, the inequalities
in \eqref{eq:criticalt} are actually strict inequalities
and each of the three possible cases occurs for some
values of $t$. 

Ultimately, $t$ will be related by equation \eqref{eq:tina}
to the strength $a$ of the fixed charges on the sphere.
The critical values $t_{1,cr}$ and $t_{2,cr}$ will
correspond to $a_{2,cr}$ and $a_{1,cr}$ (in that order)
that are used at the end of section \ref{subsec11}.

\begin{remark} \label{remark24}
	There is a symmetry between $q$ and $q^{-1}$
	that allows us to restrict attention to $0 < q < 1$.
	
	Let $\vec{\mu} = (\mu_1, \ldots, \mu_r)$
	be a vector of measures as in the VEP of 
	Definition~\ref{def21}. Let $\nu_j$ be the image
	of $\mu_j$ under the inversion $x \mapsto 1/x$,
	i.e., $\nu_j$ is the measure on $\Delta_j$
	with 
	\begin{equation} \label{eq:remark241} 
	\int f d\nu_j =  \int f\left(\frac{1}{x}\right) d \mu_j(x) \end{equation}
	for a function $f$ on $\Delta_j$.
	Then it is an easy calculation to show that
	\begin{equation} \label{eq:remark242} 
		I(\nu_j,\nu_k) = I(\mu_j,\mu_k)
		+ \int \log |x| d\mu_j(x)
		\int d\mu_k + \int \log |x| d\mu_k(x) \int d\mu_j. \end{equation}
	
	Let us use $\mathcal E_q$ to denote  the energy functional	\eqref{eq:VEPF} corresponding 
	to the parameter $q > 0$. Then using \eqref{eq:remark242} and the total
	masses \eqref{eq:massmuj} of the measures we find 
	after straightforward calculations that
	\[ \mathcal E_{1/q}(\vec{\nu})
		=  \mathcal E_q(\vec{\mu})
			-   \frac{r(r+ 2t-1)}{t^2} \log q. \]
	Thus whenever $\vec{\mu} = (\mu_1, \ldots, \mu_r)$
	is  the minimizer of the VEP with parameter $q$,
	then $\vec{\nu} = (\nu_1, \ldots, \nu_r)$
	is the minimizer with parameter $1/q$. 

\medskip
Due to this symmetry between $q$ and $q^{-1}$, 
the support of $\mu_1$ is always unbounded for $q > 1$,
since for $0 < q < 1$ the support contains $0$
by part (c) of Theorem~\ref{theorem23}.
Instead of the
$\BIS$ case, we then have the $\UIS$ case for $t$ up to the first critical value. It is then
continued with the $\TIS$ case, and after 
a second critical value with the $\FIS$ case.
	
	\medskip
	
	The  above also shows that for $q=1$, the measure $\mu_1$ is invariant under the
	inversion $x \mapsto 1/x$. Then we do not have a
	$\BIS$ or $\UIS$ case, but
	we start with a $\TIS$ case for $t$ up to 
	a critical value, followed by the $\FIS$ case.	
\end{remark}

\subsection{A meromorphic function on a Riemann surface}
\label{subsec22}

\begin{figure}[t]
	\begin{center}
		\begin{tikzpicture}[scale = 1.5,line width = 1]
		\draw [very thick] (0,1) -- (0.5,1);
		\draw [very thick] (1,1)--(1.95,1);
		\filldraw (0,1) circle (.04);
		\filldraw (0.5,1) circle (.04);
		\filldraw (1,1) circle (.04);
		\draw (0,1) node[above]{$0$};
		\draw (0.5,1) node[above]{$x_1$};
		\draw (1,1) node[above]{$x_2$};
		
		\draw (-1.2,1.3) -- (2.2,1.3);
		\draw (-2,.7) -- (1.7,.7);
		\draw (2.2,1.3) -- (1.7,.7);
		\draw (-1.2,1.3) -- (-2,.7);
		\draw (-1.6,1.1) node[left]{$\mathcal{R}^{(1)}$};
		
		\draw [very thick](-1.6,0) -- (0.5,0);
		\draw [very thick] (1,0) -- (1.95,0); 
		\filldraw (0,0) circle (.04);
		\filldraw (0.5,0) circle (.04);
		\filldraw (1,0) circle (.04);
		\draw (0.5,0) node[below]{$x_1$};
		\draw (1,0) node[below]{$x_2$};
		
		\draw (-1.2,0.3) -- (2.2,0.3);
		\draw (-2,-.3) -- (1.7,-.3);
		\draw (2.2,.3) -- (1.7,-.3);
		\draw (-1.2,0.3) -- (-2,-.3);
		\draw (-1.6,0.1) node[left]{$\mathcal{R}^{(2)}$};
		
		\draw [very thick](-1.6,-1) -- (0,-1);
		\filldraw (0, -1) circle (.04);
		\draw (-1.2,-.7) -- (2.2,-.7);
		\draw (-2,-1.3) -- (1.7,-1.3);
		\draw (2.2,-.7) -- (1.7,-1.3); 
		\draw (-1.2,-.7) -- (-2,-1.3);
		\draw (-1.6,-0.9) node[left]{$\mathcal{R}^{(3)}$};
		
		\draw (0,-1) node[below]{$0$};		
		\draw [dashed] (0,1) -- (0,-1);
		\draw [dashed] (0.5,1) -- (0.5,0);
		\draw [dashed] (1,1) -- (1,0);
		\end{tikzpicture}  	\end{center}
	\caption{ 
		The Riemann surface $\mathcal R$  in $\TIS$ case
	(for $r=2$)  \label{figure2}}
\end{figure}
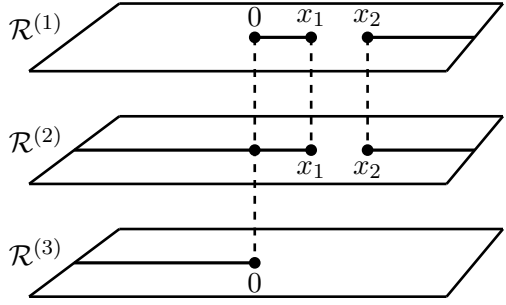

In what follows we restrict to the case $0 < q < 1$.
Then $0 \in \supp(\mu_1)$ by Theorem \ref{theorem23} (c)
and we are in one of the $\BIS$, $\TIS$ or $\FIS$ cases.
We also let $0< t < 1$.
Our further results are based on the consideration of
a Riemann surface.
\begin{definition} \label{def24} The Riemann surface
	$\mathcal R$ (see Figure \ref{figure2}) has 
	$r+1$ sheets $\mathcal R^{(j)}$, $j=1, \ldots, r+1$, where
\begin{equation}  \label{eq:Rsheets}
	\begin{aligned} 
	 \mathcal R^{(1)} & = \mathbb C \setminus \supp(\mu_1),  \\
	\mathcal R^{(j)} & = \mathbb C \setminus 
		(\supp(\mu_{j-1}) \cup \supp (\mu_j)),
		\quad \text{ for } j =2, \ldots, r+1, \\  
	\mathcal R^{(r+1)} & = \mathbb C \setminus \supp(\mu_{r})
	\end{aligned}
\end{equation}
where $(\mu_1, \ldots, \mu_r)$ is the unique minimizer
for the VEP of Definition \ref{def21}, and  we recall that $\supp(\mu_1) = \Sigma_1 \subset [0,\infty)$
and $\supp(\mu_j) = \Delta_j = (-1)^{j-1} [0,\infty)$
for $j=2,\ldots, r$. Sheet $\mathcal R^{(j)}$ is connected
to sheet $\mathcal R^{(j+1)}$ along the support of $\mu_j$
in the usual crosswise manner for $j=1, \ldots, r$. 
We also add two (in $\BIS$ case) or 
one (in other cases) points at infinity in order to 
obtain a compact Riemann surface $\mathcal R$.
\end{definition}

A count of branch points, together with the Riemann-Hurwitz formula, see e.g.\ \cite{Sc14}, 
shows that $\mathcal R$ has genus zero  in the $\BIS$ and $\FIS$ cases, while the genus is one in the $\TIS$ case.

The Stieltjes transform of the measure $\mu_j$ is
\begin{equation} \label{eq:Fj}
	F_j(z) = \int \frac{d\mu_j(x)}{z-x},  \qquad z \in \mathbb C \setminus \supp(\mu_j). \end{equation}
This is also defined for $j=0$ and $j=r+1$ in which cases
we have the simple rational functions
\begin{equation} \label{eq:F0}
	F_0(z) = \frac{-1+t}{t(z+q^{-1})},
	\qquad F_{r+1}(z) = \frac{r+t}{t(z-(-1)^r q)} 
	\end{equation}
We use the Stieltjes transforms  
to define a function on $\mathcal R$.
\begin{definition} \label{def26} 
The function $\Phi$ is defined on the Riemann surface
via its restrictions $\Phi^{(j)}$, $j=1, \ldots, r+1$, to
the various sheets, by 
\begin{equation}\label{eq:Phij}
	\Phi^{(j)}(z) = t F_j(z) - t F_{j-1}(z),
		\qquad z \in \mathcal R^{(j)},
			\end{equation}
for $j=1, \ldots, r+1$.
\end{definition}

Differentiating the identity \eqref{eq:Umuj} we obtain
\begin{equation} \label{eq:Fjjump}
	F_{j,+} - F_{j,-} = F_{j-1} + F_{j+1}
	\quad \text{on } \Delta_j = \supp(\mu_j),
	\qquad \text{for } j = 2, \ldots, r, \end{equation}
which means in view of \eqref{eq:Phij} that
$\Phi^{(j)}_{\pm} = \Phi^{(j+1)}_{\mp}$
on $\supp(\mu_j)$ for $j=2, \ldots, r$.
Thus $\Phi$ is analytic across the cut connecting
sheets $\mathcal R^{(j)}$ and $\mathcal R^{(j+1)}$
for $j \geq 2$. $\Phi$ is also analytic across the
cut connecting sheets $\mathcal R^{(1)}$ and
$\mathcal R^{(2)}$, as this follows from
the variational condition associated with the VEP 
\[ 2 U^{\mu_1} = U^{\mu_0} + U^{\mu_2} + c
	\quad \text{on } \supp(\mu_1), \]
which upon differentation leads to \eqref{eq:Fjjump}
on $\supp(\mu_j)$ for $j=1$ as well.
Thus $\Phi$ is meromorphic on $\mathcal R$
and it has a number of crucial properties that
will be discussed in section \ref{subsec41}. 

\subsection{The subset $U$} \label{subsec23}
With the help of $\Phi$ we define a subset $U$ 
of the complex plane that will lead to the droplet. 

\begin{definition} \label{def27} 
	The set $U \subset \mathbb C \cup \{\infty\}$ is
	defined by 
\begin{equation} \label{eq:defU} 
	U = \overline{\{ z \in \mathbb C \mid \left(\Im z\right) \cdot \Im 
	\left(z \Phi^{(1)}(z)\right) < 0 \}}. 
\end{equation}
We write $U_t$ if we want to emphasize the 
dependence of $U$ on  the parameter $0 < t < 1$.
\end{definition}

\begin{theorem} \label{theorem28} Let $0 < q <1$ be fixed. 
	Then the following hold.
	\begin{enumerate}
		\item[\rm (a)] $U$ is a closed set with the
		properties
		$\Sigma_1 \subset U$ and $- q^{-1}$. 	
				
		\item[\rm (b)] For $0 < t \leq t_{1,cr}$ 
		(the $\BIS$ case), $U$
		is a bounded simply connected set.
		
		\item[\rm (c)] For $t_{1,cr} < t < t_{2,cr}$ 
		(the $\TIS$ case), 
		$U$ consists of two disjoint components: a bounded component containing
		$[0,x_1]$ and an unbounded component containing $[x_2,\infty)$.
		The complement $\mathbb C \setminus U$ is a bounded doubly connected domain.
		
		\item[\rm (d)] For $t_{2,cr}  \leq t < 1$ 
		(the  $\FIS$ case), $U$ is unbounded and connected. The complement
		$\mathbb C \setminus U$ is bounded and simply connected.
		
		\item[\rm (e)] $t \mapsto U_t$ is increasing with $t$.
		
		\item[\rm (f)] $z \Phi^{(1)}(z)$ is real-valued on the boundary 	$\partial U$ and 
		\begin{equation} \label{eq:partialU}
			z \Phi^{(1)}(z) = \frac{|z|^{\frac{2}{r+1}}}{1 + |z|^{\frac{2}{r+1}}} \qquad \text{ for } z \in \partial U. \end{equation}
	\end{enumerate}
\end{theorem}
The proof of Theorem \ref{theorem28} is in section \ref{sec4}. See Figure \ref{figureU} for plots
of $U$ in the three cases.

\begin{figure}[t]
	\centering	
	\begin{overpic}[scale=0.45]{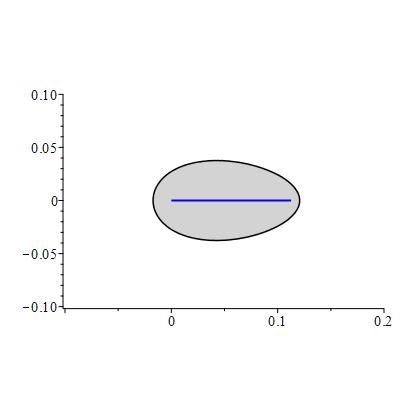}
		\put(42,45){$0$}
		\put(69,45){$x_1$}
		\put(20,45){$-q^{-1}$}
		\put(80,60){$\BIS$ case}
	\end{overpic} \\ \vspace*{-12mm}
	\begin{overpic}[scale=0.35]{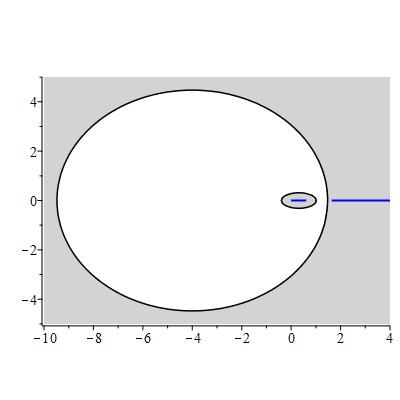}
		\put(69,42){$0$}
		\put(75,43){$x_1$}
		\put(82,52){$x_2$}
		\put(25,57){$\TIS$ case}
		\put(30,45){$-q^{-1}$}
	\end{overpic} \qquad
	\begin{overpic}[scale=0.3]{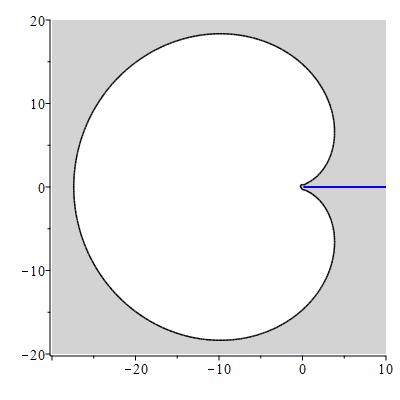}
		\put(70,48){$0$}
		\put(30,62){$\FIS$ case}
		\put(45,50){$-q^{-1}$}
	\end{overpic}
	\caption{Domain $U$ (shaded region) in the three cases. The blue  line denotes the  support of $\mu_1$
		and it is located within $U$, while $-q^{-1}$ is located outside $U$.
		 \label{figureU}}
\end{figure}
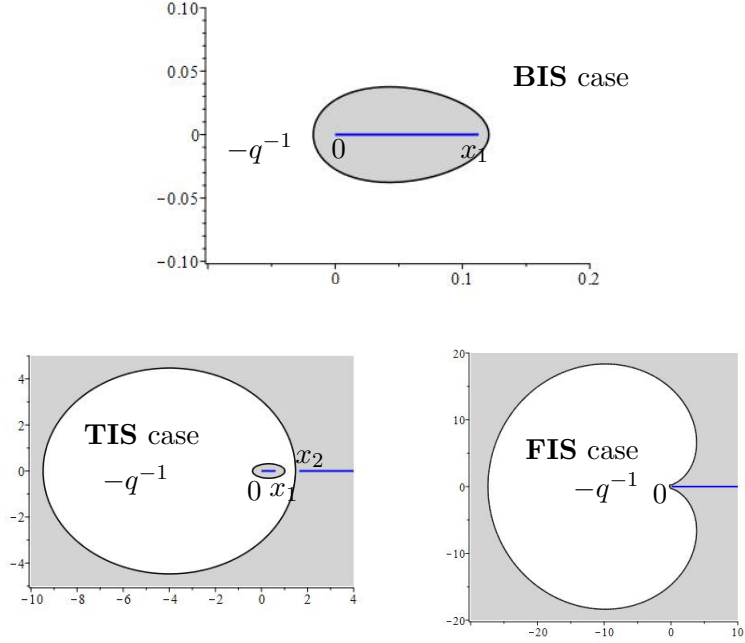

\subsection{The symmmetric domain $\Omega$
	with spherical measure}

We now introduce $r+1$ fold symmetry.

\begin{definition} \label{def29}
	We define a domain $\Omega$ (see Figure \ref{figureOmega})
	\begin{equation} \label{eq:defOmega} 
		\Omega = \{ z \in \mathbb C \mid z^{r + 1} \in U \},
		\end{equation}
	and a function
	\begin{equation} \label{eq:defS}
	S(z) = z^r \Phi^{(1)} \left(z^{r+1}\right)
	\end{equation}
	which we call the spherical Schwarz function of $\partial \Omega$.
	\end{definition}
We call $S$ the spherical Schwarz function because of the
property 
\begin{equation} \label{eq:sphSchwarz} 
 	 S(z) = \frac{\bar{z}}{1+|z|^2} \quad \text{ for } z \in \partial \Omega,
 	 \end{equation}
which follows from \eqref{eq:partialU} and the definitions in  Definition \ref{def29}. It readily follows
from \eqref{eq:sphSchwarz} that $\frac{S(z)}{1-zS(z)} = \overline{z}$
for $z \in \partial \Omega$, and so $\frac{S(z)}{1-zS(z)}$
is the usual Schwarz function of $\partial \Omega$,
and the two notions are very much intertwined,
see also the paper \cite{CrCl03} on vertex dynamics on the sphere.

Then $S$ is defined and meromorphic
on $\{ z \in \mathbb C \mid z^{r+1} \not\in \supp(\mu_1) \}$
with poles at the solutions of $z^{r+1} = - q^{-1}$,
with the behavior $z S(z) \to 1$ as $z \to \infty$.
Also $S$ has an analytic continuation 
to a meromorphic
function on a compact $r+1$ sheeted Riemann surface 
where  $S(z) = z^r \Phi^{(j)}(z^{r+1})$ on the $j$th
sheet. This analytic continuation has poles on the
$(r+1)$st sheet given by the solutions of $z^{r+1} = (-1)^r q$.

\begin{figure}[t]
	\centering	
	\begin{overpic}[scale=0.4]{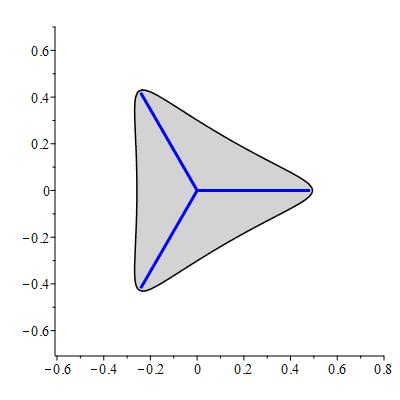}
		\put(70,70){$\BIS$ case}
	\end{overpic} \\ \vspace*{-2mm}
	\begin{overpic}[scale=0.4]{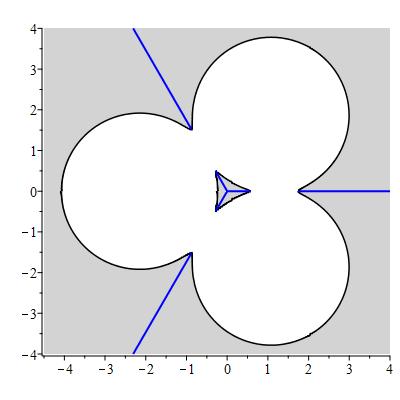}
		\put(22,55){$\TIS$ case}
	\end{overpic} \qquad
	\begin{overpic}[scale=0.4]{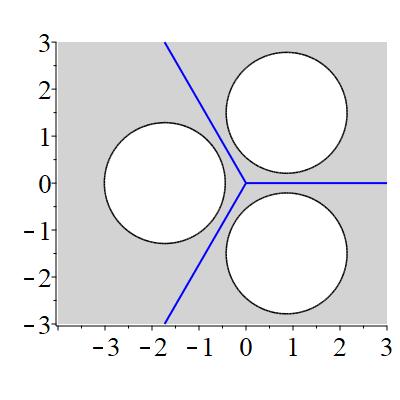}
		\put(15,80){$\FIS$ case}
	\end{overpic}
	\caption{Domain $\Omega$ (shaded region) in the three cases for $r=2$. The blue lines are the  support of $\mu^*$.
		\label{figureOmega}}
\end{figure}

We next define the  two measures $\mu_{\Omega}$ and $\mu^*$.
\begin{definition} \label{def210}
We define a measure $\mu_{\Omega}$ on $\Omega$ by
\begin{equation} \label{eq:muOmega}
	d\mu_{\Omega}(z) = \frac{1}{\pi t} \left. \frac{dA(z)}{\left( 1+ |z|^2\right)^2} \right|_{\Omega},  
	\end{equation}
and $\mu^*$ as the unique measure on the $r+1$ star
\begin{equation} \label{eq:star} 
	\{ z \mid z^{r+1} \in \mathbb R^+ \} 
	\end{equation}
that is invariant under rotation 
$z \mapsto e^{\frac{2\pi i}{r+1}} z$
and whose pushforward under $z \mapsto z^{r+1}$
is equal to $\mu_1$ (the first component of the
minimizer of the VEP).
\end{definition}

This leads to the final main result of the paper.

\begin{theorem} \label{theorem211}
	\begin{enumerate}
		\item[\rm (a)] 	$\mu_{\Omega}$ is a probability measure on $\Omega$, and $\mu^*$ is a probability
		measure on \eqref{eq:star}.
		\item[\rm (b)] $t \mapsto \Omega_t$, $t \mapsto t \mu_{\Omega,t}$ 
		and $t \mapsto t \mu^*_t$ are increasing for 
		$t \in (0,1)$. 
		\item[\rm (c)] There is a constant $c_1 = c_{1,t}$ such that
		\begin{align} \label{eq:muOmega1bis}
		U^{\mu_{\Omega}}(z) -
		\frac{1-t}{(r+1)t} \log|z^{r+1}+q^{-1}|
		+ \frac{1}{2t} \log \left( 1 + |z|^2 \right) 
		\begin{cases}
		= c_1, & z \in \Omega, \\
		\geq c_1, & z \in \mathbb C.
		\end{cases}	
		\end{align} 
	If $\Omega$ is unbounded, then $c_1 = 0$.
	\item[\rm (d)] There is a constant $c_2 = c_{2,t}$ such that
	\begin{align} \label{eq:muOmega2bis} 
	U^{\mu_{\Omega}}(z) - U^{\mu^*}(z)
	\begin{cases} 
	= c_2, & z \in \mathbb C \setminus \Omega, \\
	\leq c_2, & z \in \mathbb C.
	\end{cases}
	\end{align}
	If $\Omega$ is bounded, then $c_2 = 0$.
\end{enumerate}
\end{theorem}
The proof of Theorem \ref{theorem211} is in section \ref{sec5}.

Parts (c) and (d) of Theorem \ref{theorem211} 
tell us that the equations \eqref{eq:muOmega1} and 
\eqref{eq:muOmega2} are satisfied provided
\begin{equation}  \label{eq:aint} 
	a = \frac{1-t}{(r+1)t} \end{equation}
which agrees with \eqref{eq:tina}.
Thus, as  already explained, Theorem \ref{theorem211}  shows that 
the image of $\Omega$ under inverse stereographic projection
is the droplet $D_{\sigma}$ on the unit sphere, 
and the pullback of $\mu^*$ is the motherbody $\sigma^*$. 

\begin{remark} \label{remark212}
The domain $\Omega_t$ increases with $t$
according to part (b) of Theorem \ref{theorem211}.
It is an instance of Laplacian growth 
(or Hele-Shaw flow) in the spherical metric.
We refer to \cite{GuTeVa14} and the references
therein for more on the interesting topic of Laplacian
growth and its many connections.

Figure \ref{figureOmega} contains a plot of 
$\Omega$ in the various cases for the value $r=2$.
It is interesting to note that in the BIS
case $\Omega$ coincides with the droplet in the normal matrix model
with a cubic potential, see e.g.\
\cite{BlKu12,ElFe05,TeBeAgZaWi05}.
The eigenvalues in this random matrix model tend
to the droplet with a uniform density (in contrast
to  \eqref{eq:muOmega} which is uniform in
the spherical metric), and the zeros
of related orthogonal polynomials tend to  
the motherbody. The limiting zero counting measure 
is characterized by a vector equilibrium problem in \cite{BlKu12} that is however different from the VEP of
Definition \ref{def21}, see also \cite{KuLo15}
for the case $r \geq 3$. Our VEP can be seen as a spherical analogue from the VEPs in \cite{BlKu12,KuLo15}.

In the normal matrix model with a cubic potential the
droplet grows up to a critical time and then cusps 
appear on the boundary of the droplet that cause
a breakdown of the model, (see however \cite{KuTo15, LeTeWi09, LeTeWi10, LeTeWi11} for continuations beyond
breakdown).
In our model there is no breakdown since the transition
to the $\TIS$ takes place before we reach  the cusp situation. 
\end{remark}

\section{Proof of Theorem \ref{theorem23}} \label{sec3}

\subsection{Proof of part (a)} \label{subsec31}

\subsubsection{A more general result} \label{subsec311}
Given the second component $\mu_2$ of the solution of
the VEP of Definition \ref{def21}, $\mu_1$
is the probability measure $\mu$ on $[0,\infty)$ that
minimizes $I(\mu) - I(\mu,\sigma)$ where
$\sigma = \mu_0 + \mu_2$ is a signed measure with
integral $\int \sigma = 2$.
Part (a) of 
Theorem \ref{theorem23} will  follow from the following
more general result, where it is important
that the negative part of $\sigma$ is a Dirac point
mass. Note that $\sigma$ used in this section is 
not related to $\sigma$ from \eqref{eq:sigma}.

\begin{proposition} \label{prop31}
Suppose $\sigma = - A \delta_{-q^{-1}} + \sigma^+$
where $A > 0$ and $\sigma^+$ is a measure
on $(-\infty,0]$ with $A < \int d\sigma^+ < \infty$.
Then there is a unique $\mu$ on $[0,\infty)$ that minimizes
\[ I(\mu) - I(\mu, \sigma) \]
among all measures on $[0,\infty)$  with $\int d\mu = 
\frac{1}{2} \int \sigma$. 
The support $\Sigma = \supp(\mu)$ 
takes one of the forms described in Theorem \ref{theorem23}~(a), 
namely $\Sigma$ is a bounded interval $[0,x_1]$ containing $0$, an
unbounded interval $[x_2,\infty)$ not containing $0$, 
the disjoint union of two intervals
$[0,x_1] \cup [x_2,\infty)$, or the full half-line $[0,\infty)$.
\end{proposition}

The minimization problem in Proposition \ref{prop31} is again weakly admissible, and there is a unique minimizer
$\mu$. If we relax the condition that $\mu$ is a measure
and also allow signed measures, then the minimizer is
the balayage 
\begin{equation} \label{eq:theo2proof2}
\nu = \frac{1}{2} \Bal(\sigma, [0,\infty)) \end{equation}
which is known to have the density
\begin{equation} \label{eq:theo2proof3}
\frac{d\nu}{dx} =
\frac{1}{2\pi \sqrt{x}}
\int_{-\infty}^0 \frac{\sqrt{|s|}}{x-s} d\sigma(s),
\qquad 	0 < x < \infty. \end{equation}
If the density \eqref{eq:theo2proof3} happens
to be non-negative on $[0,\infty)$, then $\mu = \nu$
solves the minimization problem and $\supp \mu = [0,\infty)$.  

If $\nu$ is not a positive measure, then we use the
idea of iterated balayage \cite{Dr97,KuDr99}. 
This method is based on the  fact that $\mu \leq \nu^+$ where $\nu^+$ is the positive part of $\nu$ in its Jordan decomposition
\[ \nu = \nu^+ - \nu^-. \]
In particular $\supp(\mu) \subset \supp(\nu^+)$,
see \cite[Lemma 3]{KuDr99}. With this
information we can restrict the minimization problem 
to measures supported on $\supp(\nu^+)$, and if we 
also allow signed measures then the minimum is attained
by
\[ \nu^+ - \Bal(\nu^-, \supp(\nu^+)). \]
If this happens to be a positive measure then it is 
equal to $\mu$, and we can stop. Otherwise we repeat 
the above step, which leads to the following
iterative procedure.

We put $\nu_1 = \nu$, and iteratively for $k=1,2, \ldots$, we
write $\nu_k = \nu_k^+ - \nu_k^-$
where $\nu_k^+$ and $\nu_k^-$ are the positive and negative
parts of $\nu_k$, and we define
\begin{equation} \label{eq:lem32proof2} 
	\nu_{k+1} = \nu_k^+ - \Bal(\nu_k^-, \supp(\nu_k^+)). 
	\end{equation}
The convergence properties of the sequence $(\nu_k)_k$
are not fully understood, but in cases where we can control the
supports of the measures  we will have that
$\nu_k^- \to 0$ and $\nu_k^+ \to \mu$ as $k \to \infty$.

Under the conditions of Proposition \ref{prop31} we
can indeed control the supports, and we will show that 
for each $k$ the support of $\nu_k^+$ takes one of the forms
stated in the proposition, namely
$\supp(\nu_k^+)$ is either a bounded interval $[0,x_{1,k}]$,
an unbounded interval $[x_{2,k}, \infty)$, a union
of two intervals $[0,x_{1,k}] \cup [x_{2,k}, \infty)$,
or the full half-line $[0,\infty)$. 
Since the supports are decreasing if $k$ increases,
the two sequences $(x_{1,k})$ and $(x_{2,k})$ are either
finite (maybe even empty), or else they 
monotonically converge to limits
$x_1$ and/or $x_2$. In this way we will be able to show that 
$\supp(\mu)$ has one of the forms in the proposition.

In the first step we show that the support of
$\nu^+ = \nu_1^+$ has the required form. 

\subsubsection{First step: The support of $\nu^+$}
\label{subsec312}

For $\sigma$ as in the statement of Proposition \ref{prop31},
the density \eqref{eq:theo2proof3} of $\nu$  takes the form
\begin{equation} \label{eq:theo2proof4}
\frac{v(x)}{2\pi \sqrt{x}}, \quad \text{with} \quad
 v(x) = \int_{0}^{\infty} \frac{d\rho(s)}{x+s} - \frac{A_1}{x+q^{-1}},
\qquad 	0 < x < \infty, \end{equation}
where we put $A_1 = \ds \frac{A}{\sqrt{q}} > 0$ and $d\rho(s) = \sqrt{s} \, d\sigma^+(-s)$. 

\begin{lemma}  \label{lemma32}
In the above setting the following hold.
\begin{enumerate}
	\item[\rm (a)] $v$ has at most two zeros in $(0,\infty)$.
	\item[\rm (b)] There exist $0 \leq x_1 \leq x_2 \leq \infty$ such that
	\[ \supp(\nu^-) = [x_1,x_2] \quad \text{and} \quad
	\supp(\nu^+) = \overline{(0,x_1) \cup (x_2,\infty)}. \]
\end{enumerate}
\end{lemma}
\begin{proof}
(a) Suppose, to get
	a contradiction, that $0 < x_0 < x_1 < x_2 < \infty$
	are three zeros of $v$.
Let $s_1 = q^{-1}$ and write $v = f_0 - f_1 + f_2$ with
\begin{equation} \label{eq:lem31proof1} 
	f_0(x) =  \int_0^{s_1} \frac{d\rho(s)}{x+s}, \quad
	f_1(x) = \frac{A}{x+s_1}, \quad
f_2(x) = \int_{s_1}^{\infty} \frac{d\rho(s)}{x+s}. 
\end{equation}
 Then by \eqref{eq:lem31proof1} and the multilinearity of the determinant
\begin{equation} \label{eq:lem31proof2}	
\det \left[ f_k(x_j) \right]_{j,k=0}^2
		= A_1 \int_0^{s_1} d\rho(s_0) \int_{s_1}^{\infty} d\rho(s_2)
			\det \left[ \frac{1}{x_j+s_k} \right]_{j,k=0}^2. 
			\end{equation}
There is an explicit formula for the determinant
(Cauchy determinant) 
\[		\det \left[ \frac{1}{x_j+s_k} \right]_{j,k=0}^2
	= \frac{ \prod\limits_{0 \leq j<k \leq 2} (x_k-x_j) 
		\prod\limits_{0 \leq j<k \leq 2} (s_k-s_j)}{\prod\limits_{j=0}^2 \prod\limits_{k=0}^2 (x_j+s_k)}. \] 
In the integral in \eqref{eq:lem31proof2} we have $0 < s_0 < s_1 < s_2$, and since also  $0 < x_0 < x_1 < x_2$, we see that the
Cauchy determinant is $>0$. Since $A_1 > 0$ it follows
from \eqref{eq:lem31proof2}
that $\det [f_k(x_j)] > 0$ and the matrix
$[f_k(x_j)]_{j,k=0}^2$ is invertible. 

However, since $v = f_0 - f_1 + f_2$ by \eqref{eq:theo2proof4} and \eqref{eq:lem31proof1}, 
and since $v(x_j) = 0$ for $j=0,1,2$, it follows that
\[ \begin{pmatrix}
	f_0(x_0) & f_1(x_0) & f_2(x_0) \\
	f_0(x_1) & f_1(x_1) & f_2(x_1) \\
	f_0(x_2) & f_1(x_2) & f_2(x_2) \end{pmatrix}
	 \begin{pmatrix} 1 \\ -1 \\ 1 \end{pmatrix} = 
	\begin{pmatrix} v(x_0) \\ v(x_1) \\ v(x_2) \end{pmatrix} = \begin{pmatrix} 0 \\ 0 \\ 0
	\end{pmatrix}  \]
and this is a contradiction, since the matrix is invertible.

\medskip

(b) From part (a) we know that $v$ has at most two zeros
in $(0,\infty)$. By continuity, $v$ can also 
have at most two sign changes in $(0,\infty)$. 

If $v$ has no sign changes then $v \geq 0$ on $(0,\infty)$,
since due to the fact that $\sigma$ has the density
\eqref{eq:theo2proof4} and
\begin{equation} \label{eq:theo2proof5} 
\int d\sigma = \int_0^{\infty} \frac{v(x)}{2\pi \sqrt{x}} dx > 0 \end{equation}
it cannot be fully $\leq 0$.
Then we take $x_2 = x_1$ in the lemma.

If $v$ has one sign change, say at $x^* >0$, and if there
is no other zero of $v$, then $v$ is
either $> 0$ on $(0,x^*)$ and $< 0$ on $(x^*,\infty)$, or  vice versa.  
In the former case we take $x_1 = x^*$
and $x_2 = \infty$ and in the latter case we take $x_1 = 0$
and $x_2 = x^*$.
If there is another zero in $(0,\infty)$, then the
inequality is not strict at this one zero, but
we still take $x_1$ and $x_2$ as above, and the conclusion
of part (b) holds true if $v$ has one sign change. 
[It is actually not possible that there is  
another zero, but we do not need this fact.]

If $v$ has two sign changes, say at $0 < x_1 < x_2 < \infty$, then 
$v$ is either positive on $(0, x_1)$, negative  on $(x_1, x_2)$, and
positive again on $(x_2,\infty)$, or the other way around
negative  on $(0,x_1)$, positive on $(x_1,x_2)$ and negative on $(x_2,\infty)$.
[Now we can be sure that the inequalities are strict since there are no more than two zeros by part (a).]
The latter possibility cannot happen, which we can see by 
adding $\epsilon \delta_{-\epsilon^2}$ to $\nu_0$ for some small $\epsilon >0$.
Part (a) continues to apply and it follows that
\begin{equation} \label{eq:vepsilon} 
 v_{\epsilon}(x) = v(x) + \frac{\epsilon}{x+\epsilon^2} \end{equation}
has at most two sign changes on $(0,\infty)$. Since $v_{\epsilon}(0) > 0$ 
for sufficiently small $\epsilon > 0$, the set where  $v_{\epsilon} < 0$ is
then at most a single interval.
Letting $\epsilon \to 0+$, we then arrive at a contradiction in case
$v$ is negative on $(0,x_1) \cup (x_2, \infty)$.

This proves part (b) of the lemma in all cases.
\end{proof}

\subsubsection{Second step: Monotonicity  of $v$ on 
$\supp(\nu^+)$} \label{subsec313}

In order to make the induction step in the iterated
balayage argument that follows 
we need the following behavior of $v$ on the parts
where it is positive.

\begin{lemma} \label{lemma33}
Under the same conditions as in Lemma \ref{lemma32}
where we let $x_1, x_2$ be as in part (b) of Lemma \ref{lemma32}, the following hold.
\begin{enumerate}
\item[\rm (a)] $v'$ and $(xv)'$ have at most two
zeros in $(0,\infty)$.
\item[\rm (b)] If $0 < x_1 < x_2$ then $x \mapsto v(x)$
is strictly decreasing for $ x \in (0, x_1)$.
\item[\rm (c)] If $x_1 < x_2 < \infty$ then
$x \mapsto x v(x)$ is strictly increasing for $x \in (x_2, \infty)$. 
\end{enumerate}	 
\end{lemma}
\begin{proof}
(a) The proof is similar to the proof of part (a) of
Lemma \ref{lemma32}. Note that from \eqref{eq:theo2proof4}
\begin{equation} \label{eq:vprime} v'(x)
= -\int_0^{\infty} \frac{d\rho(s)}{(x+s)^2}
+ \frac{A_1}{(x+q^{-1})^2}. \end{equation}
We let $s_1 = q^{-1}$ and write $v' = f_0' - f_1'  + f_2'$ with $f_0, f_1, f_2$
as in \eqref{eq:lem31proof1} and then
\begin{equation} \label{eq:lem31proof3} 
\det \left[ f_j'(x_k) \right]_{j,k=0}^2
= - A_1 \int_0^{s_1} d\rho(s_0) \int_{s_1}^{\infty} d \rho(s_2)
\det \left[ \frac{1}{(x_k+s_j)^2} \right]_{j,k=0}^2. \end{equation} 
A Maple calculation shows that 
\begin{multline*} 
\det\left[ \frac{1}{(x_k+s_j)^2} \right]_{j,k=0}^2 
= \frac{\prod\limits_{0 \leq j < k \leq 2} (x_k-x_j)
	\prod\limits_{0 \leq j < k \leq 2} (s_j-s_i)}
{\prod\limits_{j=0}^2 \prod\limits_{k=0}^2  (x_k + s_j)^2}
\\ \times 
P(x_0, x_1, x_2; s_0, s_1, s_2) \end{multline*}
where $P$ is a homogeneous degree six polynomial in
the six variables whose coefficients (with respect to the
monomial basis) are all positive. Thus $P > 0$ when
all its arguments are $>0$, and it follows that 
\eqref{eq:lem31proof3} is negative, and in particular non-zero, 
whenever $0 < x_0 < x_1 < x_2$. Thus $v'$ cannot have more than two positive zeros.

The proof for $(xv)'$ is similar, since 
\[ (xv)'= \int_0^{\infty} \frac{sd\rho(s)}{(x+s)^2}
- \frac{A_1}{q (x+q^{-1})^2} \]
which has a similar form as \eqref{eq:vprime},
and the same argument applies.

\medskip
(b) Suppose $0 < x_1 < x_2$. Then $v(x_1) = 0$ and
$v(x) \geq 0$ for $x \in (0,x_1)$ by Lemma \ref{lemma32} (b). 
Also $v(x) \leq 0$ for $x \in (x_1, x_2)$.

Since $v(x) \to 0$ as $x \to \infty$, there
is a global negative minimum,  say at $x^* > x_1$, 
where the derivative vanishes and changes sign.
According to part (a), there is at most one other 
sign change of the derivative. If
this were in the interval $(0,x_1)$ then it would correspond to a local maximum of $v$ on the interval $(0,x_1)$. Then we modify
$v$ to $v_{\epsilon}$ as in \eqref{eq:vepsilon} in the proof of Lemma \ref{lemma32}. Part (a) applies 
to $v_{\epsilon}$ and it follows that $v_{\epsilon}'$ has
at most two sign changes. For $\epsilon > 0$ small enough
one sign change is close to $x^*$,  say at $x^*(\epsilon)$,
and $v_{\epsilon}$ has its global minimum there. Since $v_{\epsilon}'(0) < 0$ for $\epsilon > 0$ sufficiently small,
there can be no sign
change of $v_{\epsilon}'$ in $(0,x^*(\epsilon))$, and letting $\epsilon \to 0+$
we find that $v'$ has no sign change in $(0, x^*)$.
Thus $v$ is strictly decreasing in $(0,x_1)$
as claimed in part (b).

\medskip
(c) The proof for part (c) is similar. Suppose $x_1 < x_2 < \infty$, so that $v(x_2) = 0$ and $v(x) \geq 0$
for $x \in (x_2, \infty)$ by Lemma \ref{lemma32} (b). 

Since $xv(x) \to 0$ as $x \to 0$, there is a global
minimum of $x \mapsto xv(x)$, at $x^* \in (x_1, x_2)$ say,
where the derivative is zero and changes sign. There is
at most one more zero by part (a). If there were a sign
change of $(x v)'$ in  $(x_2,\infty)$, then that would
give us a maximum of $x \mapsto xv(x)$ on
$(x_2, \infty)$. We again modify $v$ to $v_{\epsilon}$
as in \eqref{eq:vepsilon}. Part (a), applied to $v_{\epsilon}$,
tells us that $(x v_{\epsilon})'$ has at most two sign
changes on $(0,\infty)$. For small $\epsilon > 0$
one sign change is close to $x^*$, say at $x^*(\epsilon)$
where $xv_{\epsilon}(x)$ has its global minimum.
Since 
\[ \lim_{x \to \infty} (x v_{\epsilon}(x))' > 0 \]
for small enough $\epsilon > 0$, this derivative then has
no sign change in $(x^*(\epsilon), \infty)$ and
therefore $x v_{\epsilon}(x)$ is strictly increasing in $(x^*(\epsilon), \infty)$. Letting $\epsilon \to 0+$ 
it follows that $x v(x)$ increases strictly  in $(x^*,\infty)$,
and a fortiori in $(x_2,\infty)$.
\end{proof}

\subsubsection{Third step: Iterated balayage} \label{subsec314}

In the final step 
we use the iterated balayage to complete the proof
of Proposition \ref{prop31}. 
We take $\nu_1 = \nu$ where $\nu$ is the signed
measure on $[0,\infty)$ with density 
\eqref{eq:theo2proof3}. 
If $\nu_1 \geq 0$ then $\nu_1 = \mu$
and we are in the full interval support ($\FIS$) case $\Sigma = [0,\infty)$.
	
In the rest of the proof we assume that $\nu_1$ is not 
a positive measure. Then
iteratively we construct the sequence $(\nu_k)_k$
as in \eqref{eq:lem32proof2}. 
Inductively we then have $\int d\nu_k = \int d\nu_1
= \frac{1}{2}  \int \sigma$ and 
$\mu \leq \nu_k^+$ for every $k$, and 
in particular
\[ \Sigma \subset \supp(\nu_k^+). \]
The sequence $(\int d\nu_k^+)_k$ decreases and if
$\int d\nu_k^-$ tends to $0$ as $k \to \infty$,
then  $\nu_k^+ \to \mu$ in the sense
of weak$^*$ convergence of measures on $[0,\infty]$.
In the present situation (with the help of Lemmas \ref{lemma32} and \ref{lemma33})
we can prove that this is indeed the case.

\begin{lemma} \label{lemma34}
	For every $k$ we have   
	\begin{enumerate}
		\item[\rm (a)] 
		$\supp(\nu_k^+) = \overline{(0, x_{1,k})}	
		\cup \overline{(x_{2,k}, \infty)}$ for some 
		$0 \leq x_{1,k} < x_{2,k} \leq \infty$, while
		$\supp(\nu_k^-) \subset [x_{1,k}, x_{2,k}]$, and
		\item[\rm (b)] $\nu_k$ has a density $\frac{v_k(x)}{2 \pi \sqrt{x}}$
		where $x \mapsto v_k(x)$ strictly decreases on $(0, x_{1,k})$
		and $x \mapsto x v_k(x)$ strictly increases on $(x_{2,k}, \infty)$. 
	\end{enumerate} 
\end{lemma}	

	Assuming that Lemma \ref{lemma34} holds, we complete the proof 	of Proposition \ref{prop31}
	as follows. The measures $(\nu_k^+)$ converges to $\mu$,
	and $\supp(\mu) = \overline{(0,x_1)} \cup \overline{(x_2,\infty)}$ where $x_1 = \lim_k x_{1,k}$
	and $x_2 = \lim_k x_{2,k}$. This establishes 
	Proposition \ref{prop31}.

\begin{proof}[Proof of Lemma \ref{lemma34}.]	
	For $k=1$, the statements (a) and (b) are contained in Lemmas \ref{lemma32} and \ref{lemma33}.
	
	Suppose the lemma holds for a certain $k \geq 1$.
	Let us assume that $0 < x_{1,k} < x_{2,k} < \infty$.
	We use the fact that the balayage of a delta mass $\delta_t$
	at $t \in (x_{1,k}, x_{2,k})$ onto
	$[0, x_{1,k}] \cup [x_{2,k}, \infty)$ has the density
	\[  \frac{c(t)}{2\pi \sqrt{x}} \frac{x+b(t)}{|x-t|}  \frac{1}{\sqrt{(x-x_{1,k})(x-x_{2,k})}}  \]
with positive constants $b(t) > 0$ and $c(t) >0$.
Then $\Bal(\nu_k^-, \supp(\nu_k^+))$ has the density
\[  \frac{1}{\pi \sqrt{x}} \frac{1}{\sqrt{(x-x_{1,k})(x-x_{2,k})}}
 \int c(t) \frac{x+b(t)}{|x-t|}
	 d\nu_1^-(t),
	\qquad x \in (0, x_{1,k}) \cup (x_{2,k},\infty). \]
In view of \eqref{eq:lem32proof2} and the induction hypothesis
we then obtain that $\nu_{k+1}$ has the density 
$\frac{v_{k+1}(x)}{2\pi \sqrt{x}}$ with
\begin{equation} \label{eq:lem32proof3} v_{k+1}(x) = v_k(x) - \frac{1}{\sqrt{(x-x_{1,k})(x-x_{2,k})}}
\int_{x_{1,k}}^{x_{2,k}} c(t) \frac{x+b(t)}{|x-t|}
d\nu_k^-(t),
\end{equation}
for $x \in (0, x_{1,k}) \cup (x_{2,k},\infty)$.
Also by the induction hypothesis $v_k(x)$ is strictly decreasing
on $(0,x_{1,k})$. The other term in the right-hand side
of \eqref{eq:lem32proof3} (including the minus-sign) 
is also decreasing on $(0,x_{1,k})$,
since each of the factors
$\frac{1}{\sqrt{x_{1,k}-x}}$, 
$\frac{1}{\sqrt{x_{2,k}-x}}$, 
$x + b(t)$, and $\frac{1}{t-x}$
is positive and strictly increasing for $x \in (0,x_{1,k})$ 
as $0 < x_{1,k} < t < x_{2,k}$ and $b(t) > 0$.
Also note that $c(t) d\nu_k^-(t)$ is a positive measure
on $[x_{1,k}, x_{2,k}]$. 

Thus $v_{k+1}$ is strictly decreasing on $(0, x_{1,k})$. Then 
it is either fully negative on $(0, x_{1,k})$, in which
case we take $x_{1,k+1} = 0$, or $v_{k+1}$ is positive on
some interval $(0, x_{1,k+1})$  with $0 < x_{1,k+1} < x_{1,k}$
and $v_{k+1}$ is negative on $(x_{1,k+1},x_{1,k})$.

Similar arguments show that $x v_{k+1}(x)$ is strictly increasing
on $(x_{2,k}, \infty)$. Here we need to observe that
each of the factors
$\frac{\sqrt{x}}{\sqrt{x-x_{1,k}}}$,
$\frac{\sqrt{x}}{\sqrt{x_{2,k}-x}}$, 
and $\frac{x + b(t)}{t-x}$ decreases on $(x_{2,k},\infty)$. 
Thus $v_{k+1}$ is either fully
negative there, in which case we put $x_{2,k+1} = \infty$,
or $v_{k+1}$ is positive on
some interval $(x_{2,k+1}, \infty)$  with $x_{2,k} < x_{2,k+1} < \infty$
and $v_{k+1}$ is negative on $(x_{2,k},x_{2,k+1})$.

Parts (a) and (b) of the lemma are thus proved for $k+1$
in case $0 < x_{1,k} < x_{2,k} < \infty$. If $x_{1,k} = 0$
or $x_{2,k} = \infty$, then there is an analogous reasoning
(which is simpler). 
The lemma follows by induction. 
\end{proof}

\subsection{Proof of part (b)} \label{subsec32}

\subsubsection{Definitions} \label{subsec321}
We first define maps $M$, $\widetilde{M}$ and $M_j$
between signed measures and vectors of signed measures that will 
be used in the proof of part (b) of Theorem \ref{theorem23}. 

\begin{definition} \label{def35}
	\begin{enumerate}
	\item[(a)] 
	For a signed measure $\sigma$ on $(-\infty,0]$ with
	$0 < \int d \sigma < \infty$ we define 
	$M(\sigma) = \mu$ as the measure on $[0,\infty)$
	that minimizes 
	\begin{equation} \label{eq:theo2proof1} 
	I(\mu) - I(\mu,\sigma) \end{equation}
	among $\mu \geq 0$ with $\int d\mu = \frac{1}{2} \int d\sigma$.
	\item[(b)] Similarly, for a signed measure $\sigma$ on $[0,\infty)$ with
	$0 < \int d \sigma < \infty$ we define 
	$\widetilde{M}(\sigma) = \mu$  as the measure on $(0,-\infty]$
	that minimizes \eqref{eq:theo2proof1} 
	among $\mu \geq 0$ with $\int d\mu = \frac{1}{2} d \sigma$.
	\item[(c)] 
	Consider vectors $\vec{\nu} = (\nu_0, \ldots, \nu_{r+1})$
	of signed measures of length $r+2$, such that $\nu_j$ is supported 
	on $(-1)^j [0,\infty)$ for $j=0, 1, \ldots, r+1$
	and $0 < \int d\nu_{j-1} + \int d \nu_{j+1} < \infty$ for
	$j= 1, \ldots, r$. For such $\vec{\nu}$ we define
	\begin{equation} \label{eq:Mj} 
		M_j \vec{\nu}
		= (\nu_0, \ldots, \nu_{j-1}, \widehat{\nu}_j,
			\nu_{j+1}, \ldots, \nu_{r+1}),
			\qquad j =1, \ldots, r,
			\end{equation}
	where
	\[ \widehat{\nu}_j = \begin{cases}
		M(\nu_{j-1}+\nu_{j+1}), & \text{ if $j$ is odd}, \\
		\widetilde{M}(\nu_{j-1} + \nu_{j+1}), & 
		\text{ if  $j$ is even},
	\end{cases} \]
	with $M$ and $\widetilde{M}$ as defined in parts (a)
	and (b).
	\end{enumerate}
\end{definition}

Some remarks are in order.

\begin{remark} \label{remark36}
	\begin{enumerate}
	\item[(a)]
The measures $M(\sigma)$ and $\widetilde{M}(\sigma)$
in parts (a) and (b) of Definition \ref{def35}  are minimizers of weakly
admissible equilibrium problems. The minimizers 
uniquely exist \cite{HaKu12}.
\item[(b)]  
If $\sigma \geq 0$ then $\mu = M(\sigma)$ is the
balayage measure $\mu = \frac{1}{2} \Bal(\sigma, [0,\infty))$. 
In this case we  have a monotonicity result
\begin{equation} \label{eq:Mmonotone} 
	0 \leq \sigma \leq  \widetilde{\sigma} 
	\implies M(\sigma) \leq M(\widetilde{\sigma}) \end{equation}
for measures $\sigma$ and $\widetilde{\sigma}$ on $(-\infty,0]$.
\item[(c)] 
Similarly 
\begin{equation} \label{eq:Mhatmonotone} 
	0 \leq \sigma \leq  \widetilde{\sigma} 
\implies \widehat{M}(\sigma) \leq \widehat{M}(\widetilde{\sigma}) 
\end{equation}
for measures $\sigma$ and $\widetilde{\sigma}$ on $[0, \infty)$.

\item[(d)] The maps are positive homogeneous in the sense
that $M(c\sigma) =  c M(\sigma)$, $\widehat{M}(c\sigma)
= c\widehat{M}(\sigma)$ and $M_j(c\vec{\nu})
	= c M_j(\vec{\nu})$ if $c > 0$.

\item[(e)] If $(\mu_{1,t}, \ldots, \mu_{r,t})$ is the
solution of the VEP of Definition \ref{def21} 
for some $q > 0$ and $t \in (0,1)$, 
and $\vec{\mu}_t = (\mu_{0,t}, \mu_{1,t}, \ldots, \mu_{r,t}, \mu_{r+1,t})$,  then 
\begin{equation} \label{eq:Mjfixed} 
	M_j(\vec{\mu}_t) = \vec{\mu}_t,
	\qquad  j= 1, \ldots, r.
	\end{equation}
That is,  $\vec{\mu}_t$ is a common fixed point for the mappings $M_j$. It is the only common fixed point 
among vectors $\vec{\mu}$ with $\mu_0$ and $\mu_{r+1}$
given by \eqref{eq:mu0}.
\end{enumerate}
\end{remark}

\subsubsection{Monotonicity of $M$}

We are going to apply $M$ only to positive measures and to
signed measures whose negative part is a single
point mass at $-q^{-1}$ (as in Proposition \ref{prop31}).  We need the extension
of the monotonicity result \eqref{eq:Mmonotone}
to such signed measure. It could be that the monotonicity
result is valid more generally, but we do not consider
it here since this is all we  need for our present
purposes. 

For such signed measures $\sigma$ we have the information
about the supports of $M(\sigma)$ from Proposition 
\ref{prop31}, and we also rely on the iterated balayage
that was used in the proof of Proposition \ref{prop31}.

\begin{lemma} \label{lemma37}
	Let $\sigma \leq  \widetilde{\sigma}$ be signed measure on $(-\infty,0]$
	with $0 < \int d\sigma < \int d\widetilde{\sigma} < \infty$ whose negative
	parts are single point masses at $-q^{-1}$ only. 
	Then $M(\sigma) \leq M(\widetilde{\sigma})$.
\end{lemma}
\begin{proof} 
	Under the assumptions
	of the lemma, the signed measures take the form 
	$\sigma = - A \delta_{-q^{-1}}  + \sigma^+$
	and 
	$\widetilde{\sigma}  = - \widetilde{A} \delta_{-q^{-1}} + \widetilde{\sigma}^+$
	with $A \geq \widetilde{A} \geq 0$, 
	and $0 \leq \sigma^+ \leq \widetilde{\sigma}^+$.
	We write $\mu = M(\sigma)$ and $\widetilde{\mu}
	= M(\widetilde{\sigma})$.
	
	We recall the iterated balayage algorithm  
	from the proof of Proposition~\ref{prop31}, see in
	particular Lemma \ref{lemma34}, and we apply
	it to the signed measure $\widetilde{\sigma}$. 
	That is, we start with $\nu_1 = \frac{1}{2} \Bal(\widetilde{\sigma}, [0,\infty))$, and from there we construct the sequence
	$(\nu_k)_k$ inductively by
	\[ \nu_{k+1} = \nu_k^+ - \Bal(\nu_k^-, \supp(\nu_k^+)),
		\qquad k = 1,2, \ldots. \]
	Then $(\nu_k)$ converges to $\widetilde{\mu} 
	= M(\widetilde{\sigma})$ as
	was shown in the proof of Lemma \ref{lemma34}. 
	
	Next we define a second sequence $(\rho_k)_k$ by
	$\rho_1 = \frac{1}{2} \Bal(\sigma, [0, \infty))$,
	and
	\begin{equation} \label{eq:lem41proof} 
	\rho_{k+1} = \rho_k^+ - \Bal(\rho_k^-, \supp(\nu_k^+)),
	\qquad k=1,2, \ldots. \end{equation}
	Since $\sigma \leq \widetilde{\sigma}$ we have
	$\rho_1 \leq \nu_1$, and then by induction it easily follows that $\rho_k \leq \nu_k$ for every $k$.
	Note that we deviate from the earlier construction by
	taking in \eqref{eq:lem41proof} the balayage of $\rho_k^+$  onto $\supp(\nu_k^+)$ and not onto $\supp(\rho_k^+)$. 
	Since $\supp(\nu_k^+) \supset \supp(\rho_k^+)$, we however
	still find (by induction) that 
	$\mu = M(\sigma) \leq \rho_k^+$
	for every $k$. Then $\rho_{\infty} = \lim\limits_{k \to \infty} \rho_k$
	is a signed measure with $\rho_{\infty} \leq \widetilde{\mu}$ and
	$\mu \leq \rho_{\infty}^+$. 
	Thus $\mu \leq \widetilde{\mu}$ as claimed in the lemma.
\end{proof}

\subsubsection{$M$-convexity} \label{subsec323}

We need two more definitions. Note that $M$-convexity
is not a standard terminology, but it is introduced
here to help the exposition.

\begin{definition} \label{def38}
	Let $\vec{\nu}$ be as in Definition \ref{def35}~(c),
	and let $M_j$ be as in \eqref{eq:Mj}. Then we say that $\vec{\nu}$ is $M$-convex if
	\[ \vec{\nu} \leq M_j(\vec{\nu})
		\qquad \text{for every } j=1, \ldots, r. \]
\end{definition}

\begin{definition} \label{def39}
	The set $\mathcal M_q$ contains those
	vectors $\vec{\nu} = (\nu_0, \nu_1, \ldots, \nu_{r+1})$
	satisfying
	\begin{itemize}
		\item $\nu_j$ is a positive measure on $(-1)^{j-1} [0,\infty)$ for $j=1, \ldots, r+1$,
		\item $\nu_0 = - A\delta_{-q^{-1}}$ for 
		some $A  < \int d\nu_2$.
	\end{itemize}
\end{definition}

Then we have the following properties.

\begin{lemma} \label{lemma310} Suppose $\vec{\nu} \in \mathcal M_q$. 
\begin{enumerate}
	\item[\rm (a)] Then $M_j(\vec{\nu}) \in \mathcal M_q$
	for every $j=1,\ldots, r$.
	
	\item[\rm (b)] If $\vec{\nu} \leq \vec{\rho} \in \mathcal M_q$
	then $M_j(\vec{\nu}) \leq M_j(\vec{\rho})$
	for every $j=1,\ldots, r$.
	
	\item[\rm (c)] If $\vec{\nu}$ is $M$-convex
	then so is $M_j \vec{\nu}$ for every $j=1,\ldots, r$.
	
	\item[\rm (d)] If $\vec{\nu}$ is $M$-convex
	and $c > 0$ then $\vec{\nu} +  (c \delta_{-q^{-1}}, 0, 0, \ldots, 0)$ 	is $M$-convex.
	
	\item[\rm (e)] If $\vec{\nu}$ is $M$-convex
	and $\rho \geq 0$ is a measure on $(-1)^r [0,\infty)$ 
	then $\vec{\nu} +  (0, 0, \ldots, 0, \rho)$
	is $M$-convex.
\end{enumerate}
\end{lemma}

\begin{proof}
(a) Obvious.

(b) This follows from the monotonicity of $M$ and $\widetilde{M}$ on positive measures, see \eqref{eq:Mmonotone}
and \eqref{eq:Mhatmonotone}, and the monotonicity of $M$
on signed measures whose negative part only contains
a point mass at $-q^{-1}$, see Lemma \ref{lemma37}.

(c) Since $\vec{\nu}$ is $M$-convex we have 
	$\vec{\nu} \leq M_k\vec{\nu}$ for every $k$.
	The maps $M_k$ and $M_j$ commute if $|j-k| \neq 1$.
	Thus it follows from part (b) that 
	\[ M_j\vec{\nu} \leq M_j M_k\vec{\nu}
		= M_k M_j \vec{\nu}, \qquad k \neq \{j-1,j+1\}.
		 \]
	For $k \in \{j-1,j+1\}$ we can verify by direct
	inspection that the inequality between
	$M_j \vec{\nu}$ and $M_k M_j\vec{\nu}$ also holds.
	The two vectors only differ at positions $k = j \pm 1$, 
	which for $M_j \vec{\nu}$
	is equal to $\nu_{j\pm 1}$, and for $M_{j\pm 1} M_j \vec{\nu}$
	it is $M( \nu_{j\pm 2} + \widehat{\nu}_j)$
	or $\widetilde{M} (\nu_{j\pm 2} + \widehat{\nu}_j )$
	(depending on the parity of $j$)
	with $\widehat{\nu}_j$ as in \eqref{eq:Mj}.
	 By $M$-convexity of $\vec{\nu}$ we have $\nu_j \leq \widehat{\nu}_j$
	 and by the monotonicity properties of $M$ and $\widetilde{M}$ and $M$-convexity once more, we have
	 \[ \nu_{j-1} \leq  M(\nu_{j\pm 2}+ \nu_j) \leq M(\nu_{j\pm 2} + \widehat{\nu}_j )  \quad \text{ if $j$ is odd} \]
	 with $M$ replaced by $\widetilde{M}$ if $j$ is even.
	This proves $M_j \vec{\nu} \leq M_k M_j \vec{\nu}$
	also in case $|j-k| = 1$ and part (c) follows.
	 
(d)	and (e) are straightfoward verifications.
\end{proof}

\subsubsection{Proof of Theorem \ref{theorem23}  (b)} 
\label{subsec324}
\begin{proof}
	Let us take $0 < s < t < 1$. Then we have to show that
	$s \mu_{j,s} \leq t \mu_{j,t}$ for every $j=0,1, \ldots, r+1$.
	This is clear for $j=0$ and $j=r+1$ due to 
	the definitions \eqref{eq:mu0}. 
	
	Write
	$\vec{\mu}_s = (\mu_{0,s}, \mu_{1,s}, \ldots, \mu_{r+1,s})$
	and similarly for $\vec{\mu}_t$.
	Then by the definition of the operators
	$M_j$, we have
	\[  M_j(s \vec{\mu}_s) = s \vec{\mu}_s,
		\quad M_j(t\vec{\mu}_t) = t \vec{\mu}_t. \]
	see also Remark \ref{remark36} (d) and (e).
		
	Now we put
	\begin{align} \nonumber \vec{\nu}_1 
	& = s\vec{\mu}_s
		+ ((t-s) \delta_{-q^{-1}}, 0, \ldots, 0,  
			(t-s) \delta_{(-1)^r q} ) \\
	& = (t \mu_{0,t}, s \mu_{1,s}, s \mu_{2,s},
		\ldots, s \mu_{r,s}, t \mu_{r+1,t}). \label{eq:vecnu0}
		\end{align}
	This is the vector $s \vec{\mu}_s$ with the $0$th
	and $r+1$st components replaced by those of $t \vec{\mu}_t$.
	Then $s \vec{\mu}_s \leq \vec{\nu}_1$ and
	$\vec{\nu}_1$ is $M$-convex by Lemma \ref{lemma310}
	(d) and (e) and the fact that $s\vec{\mu}_s$ is
	$M$-convex.
	
	We choose an infinite sequence $(j_k)_k$ in $\{1, \ldots, r\}$
	where we make sure that every $j$ in $\{1, \ldots, r\}$ appears
	an infinite number of times in the sequence. Then 
	we define a sequence $(\vec{\nu}_k)_k$ by 
	\[ \vec{\nu}_{k+1} = M_{j_k}(\vec{\nu}_k) \quad \text{ for } k = 1,2, \ldots. \]
	Inductively we find that each $\vec{\nu}_k$ is $M$-convex by Lemma \ref{lemma310}(c) and the fact that $\vec{\nu}_1$
	is $M$-convex.
	Then $\vec{\nu}_k \leq \vec{\nu}_{k+1}$ for every $k$.
	
	Also by induction it is easy to show that 
	\[ \int d(\vec{\nu}_k)_j \leq t + j-1 \]
	for every $k$ and for every $j$. 	
	Thus the sequence $(\vec{\nu}_k)$ is increasing
	with a componentwise limit 
	$\vec{\nu}_k \to \vec{\nu}_{\infty}$, as 
	$k \to \infty$ (with convergence in weak$^*$-sense).
	
	If $j = j_k$ is even, then 
	\[ (\vec{\nu}_{k+1})_j  = M \left((\vec{\nu}_k)_{j-1} + (\vec{\nu}_k)_{j+1} \right). \]
	while for $j$ is odd we have to replace $M$ by $\widetilde{M}$.
	
	If we take the limit $k \to \infty$ along the subsequence for which $j_k = j$, 	then 
	it follows from this that 
	$M_j(\vec{\nu}_{\infty}) = \vec{\nu}_{\infty}$
	for every $j$. Since the $0$th and $r+1$st components
	are those of $t \vec{\mu}_t$, we  
	we conclude that $\vec{\nu}_{\infty} = t \vec{\mu}_t$,
	see also Remark \ref{remark36} (e).

We combine the inequalities to find
	\[ s \vec{\mu}_s \leq \vec{\nu}_1 \leq \cdots \leq \vec{\nu}_k \leq \cdots \leq \vec{\nu}_{\infty} = t \vec{\mu}_t \]
	which indeed shows that $s \mu_{j,s} \leq t \mu_{j,t}$ for every $j$.
\end{proof}

\subsection{Proof of part (c)} \label{subsec33}

\subsubsection{An equivalent equilibrium problem}
\label{subsec331}

For $\theta > 0$ and a measure $\mu$ on $[0,\infty)$
we write
\begin{equation} \label{eq:Itheta} 
	I_{\theta}(\mu) = \iint \log \frac{1}{|x^\theta - y^{\theta}|} d \mu(x) d\mu(y) \end{equation}
which we may call the $\theta$-energy of $\mu$. For
$\theta = 1$ it reduces to the usual logarithmic
energy $I(\mu)$ of $\mu$.

\begin{proposition} \label{prop311}
	Let $(\mu_1, \ldots, \mu_r)$ be the solution
	of the VEP of Definition \ref{def21} with parameters $q > 0$ and $0 < t < 1$.
	Then the first component $\mu_1$ minimizes
	\begin{align} \label{eq:MBEP} 
	\frac{1}{2} I(\nu) + \frac{1}{2} I_{\theta}(\nu) 
		+ \int V(x) d\nu(x) \end{align}
	with $\theta = \frac{1}{r}$ and 
	\begin{equation} \label{eq:Vx} V(x) = - \frac{1-t}{t} \log\left(x+q^{-1}\right) 
		+ \frac{r+t}{t} \log\left(x^{1/r} + q^{1/r}\right) \end{equation}
	among all probability measures $\nu$ on $[0,\infty)$.
\end{proposition}
Energy functionals of the form \eqref{eq:MBEP} appeared before  in the context of Muttalib-Borodin ensembles
\cite{Bo99, Mu95}. These are joint probability
densities for $n$ particles on the positive
real line of the form
\begin{equation} \label{eq:MBensembles}
\frac{1}{Z_n} \prod_{1\leq i<j \leq n} (x_i-x_j) (x_i^{\theta} - x_j^{\theta})
\prod_{j=1}^n e^{-nV(x_j)}, \qquad \text{ all } 
x_j > 0 \end{equation}
where $\theta > 0$ is a positive constant and 
$Z_n$ is a normalization factor.
In the large $n$ limit the particles are distributed
according to the minimizer of \eqref{eq:MBEP}, see
\cite{ClRo14, FoLiZi15} and see \cite{ClGiSt19,FoWa17,KuMo19,Mo20+} for some
recent contributions on Muttalib-Borodin
ensembles \eqref{eq:MBensembles}. We use the characterization
of $\mu_1$ via the equilibrium problem from Proposition~\ref{prop311} 
in the proofs of parts (c) and (d) of Theorem \ref{theorem23}.

The proof  of Proposition \ref{prop311} follows
along the lines of the proof of Theorem 1.1 in  \cite{Ku16}.

\begin{proof}[Proof of Proposition \ref{prop311}.]
For a probability measure $\nu$ on $[0,\infty)$ 
we define
\begin{equation} \label{eq:Jnu} 
 J(\nu)  = \min_{\nu_2, \ldots, \nu_r}
	  	 \left( \sum_{j=2}^r I(\nu_j) - \sum_{j=1}^r
	  	 I(\nu_j, \nu_{j+1}) \right)	
	\end{equation}
where $\nu_1 = \nu$ and $\nu_{r+1} = \mu_{r+1} =
\frac{r+t}{t} \delta_{(-1)^r q}$.
The  minimization is over all $\nu_2, \ldots, \nu_r$
satisfying the support condition \eqref{eq:Deltaj} and
the total mass condition \eqref{eq:massmuj}, i.e.,
$\supp(\nu_j) \subset \Delta_j$ and $\nu_j(\Delta_j)
= 1 + \frac{j-1}{t}$ for $j=2, \ldots, r$.

This is again a weakly admissible vector
equilibrium problem, similar to the VEP from Definition \ref{def21}, 
and it has a unique solution. It is simpler to solve, since only 
positive measures are involved and we can be sure that the 
minimizers $\nu_2, \ldots, \nu_r$ have full supports, with the property 
\begin{equation} \label{eq:Unuj} 
	2 U^{\nu_j} = U^{\nu_{j-1}} + U^{\nu_{j+1}}
	\quad \text{on } \Delta_j, 
	\quad \text{ for } j= 2, \ldots, r, \end{equation}
see also \eqref{eq:Umuj}. 

From \eqref{eq:Unuj} we obtain 
\begin{align*} I(\nu_j) & = \int U^{\nu_j} d\nu_j  = \frac{1}{2} \int \left(U^{\nu_{j-1}} - U^{\nu_{j+1}}
	\right) d\nu_j \\
	& = \frac{1}{2} I(\nu_{j-1}, \nu_j) +  \frac{1}{2} I(\nu_j, \nu_{j+1}), \qquad j = 2, \ldots, r. \end{align*}
Hence from \eqref{eq:Jnu}
\begin{align} \nonumber 
J(\nu)  & = - \frac{1}{2} I(\nu, \nu_2) - \frac{1}{2} I(\nu_r, \nu_{r+1}) \\ \label{eq:Jnu2} 
	& = - \frac{1}{2} \int U^{\nu_2} d\nu	
	- \frac{r+t}{2t} U^{\nu_r}((-1)^r q), 
\end{align}
where $\nu_2$ and $\nu_r$ are from the minimizer
$(\nu_2, \ldots, \nu_r)$ associated with $\nu$. We are going to calculate $U^{\nu_2}$
and $U^{\nu_r}$.

We first do this for a point mass $\nu = \delta_p$ 
with $p > 0$, and the general case is obtained by averaging over $p$.
So let $(\nu_2, \ldots, \nu_r)$ be the minimizer
associated with $\nu = \nu_1 =\delta_p$. 
We use the Riemann surface 
$\mathcal S$ with $r$ sheets $\mathcal S^{(j)}$, $j=1, \ldots, r$, given by
\begin{equation} \label{eq:RSS} 
	\begin{aligned} 
	\mathcal S^{(1)} & = \mathbb C \setminus (-\infty,0], \\
	\mathcal S^{(j)} & = \mathbb C \setminus \mathbb R, && \text{for } j=2, \ldots, r-1, \\
	\mathcal S^{(r)} & = \mathbb C \setminus 
	\left((-1)^r [0,\infty)\right).
	\end{aligned} \end{equation}
Sheet $S^{(j)}$ is connected to sheet $S^{(j-1)}$ along the cut $\Delta_j = (-1)^j [0,\infty)$
for $j=2, \ldots, r$. We add a point at infinity to 
obtain a compact Riemann surface.

We define a function $\Psi$ on $\mathcal S$ by its restriction $\Psi^{(j)}$
to the $j$th sheet as follows.
\begin{equation} \label{eq:PsionS} \begin{aligned}
	\Psi^{(1)}(z) &  = \frac{z}{z-p} - z \int \frac{d\nu_2(s)}{z-s}, \\
	\Psi^{(j)}(z) & =  z \int \frac{d\nu_j(s)}{z-s} -  z \int \frac{d\nu_{j+1}(s)}{z-s},
		\quad \text{for } j=2, \ldots, r-1, \\
	\Psi^{(r)}(z) & =  z \int \frac{d\nu_r(s)}{z-s} - \frac{r+t}{t} \frac{z}{z-(-1)^rq}.
	\end{aligned} \end{equation} 
The conditions \eqref{eq:Unuj} imply that $\Psi$ is meromorphic on $\mathcal S$
with poles at $z=p$ on the first sheet and at $z=(-1)^r q$ on the $r$th sheet, see also the discussion
after Definition \ref{def26} that shows why 
$\Phi$ is meromorphic on $\mathcal R$. 
[The construction of $\Psi$ is similar to that of $\Phi$.]
Due to the total masses of the measures we have $\Psi(z) 
\to - \frac{1}{t}$ as $z \to \infty$. 

The Riemann surface \eqref{eq:RSS} has a simple parametrization $z = w^r$, and in the $w$ variable
the poles are at $w= p^{1/r}$ and $w = - q^{1/r}$. Taking into account the residues 
at the poles and the behavior at infinity, we find that
\begin{equation} \label{eq:Psi} 
	\Psi(z) = \frac{1}{r} \frac{p^{1/r}}{w-p^{1/r}} + \frac{r+t}{rt} 
	\frac{q^{1/r}}{w + q^{1/r}} - \frac{1}{t},
	\qquad z =  w^r.
	\end{equation} 
Observe also that $\Psi(z) = 0$ for $z=w=0$.	

Specifying \eqref{eq:Psi} to the first sheet, and
recalling \eqref{eq:PsionS} we find
\begin{align} \nonumber  \int \frac{d\nu_2(s)}{z-s} 
	& = \frac{1}{z-p} - \frac{1}{z} \Psi^{(1)}(z) \\
	& = \frac{1}{z-p} - \frac{1}{rz} \frac{p^{1/r}}{z^{1/r}-p^{1/r}}
		- \frac{r+t}{rtz} \frac{q^{1/r}}{z^{1/r} + q^{1/r}}
		- \frac{1}{tz} \end{align}
with principal branch of the fractional powers.
We integrate  with respect to $z$ and
find after straightforward calculation
\begin{multline} \label{eq:logpotnu2}
	\int \log(z-s) d\nu_2(s) 
	= \log(z-p) - \log\left(z^{1/r} - p^{1/r}\right)
		+ \frac{r+t}{t} \log \left(z^{1/r} + q^{1/r}
		\right) \end{multline}
There is no constant of integration since both sides
behave as $(1 + t^{-1}) \log z + o(1)$
as $z \to \infty$.
The real part of \eqref{eq:logpotnu2} gives us
the logarithmic potential 
\begin{equation} \label{eq:pUnu2} 
	U^{\nu_2}(z) = \log \left| \frac{z^{1/r} - p^{1/r}}{z-p}
\right|  - \frac{r+t}{t} \log \left| z^{1/r} + q^{1/r} \right|. \end{equation}

An analogous calculation,  based on \eqref{eq:Psi} and 
the expression
\eqref{eq:PsionS} of $\Psi$ on the $r$th sheet,
leads to the logarithmic potential of $\nu_r$,
\begin{equation} \label{eq:pUnur} 
U^{\nu_r}(z) = \frac{r+t}{r} \log \left| \frac{z^{1/r} + q^{1/r}}{z-(-1)^r q}
\right|  - \log \left| z^{1/r} - p^{1/r} \right|
\end{equation}
with the branch of the $r$th root that is analytic 
on $\mathbb C \setminus \Delta_r$ and that is real 
and negative for real $z \in \mathbb C \setminus \Delta_r$.
However, we emphasize that $p^{1/r}$ and $q^{1/r}$ always
denote the positive $r$th roots. Similarly when we
write $x^{1/r}$ with $x > 0$ as for example in \eqref{eq:MBEP}
and in \eqref{eq:Unur} below.

Thus \eqref{eq:pUnu2} and \eqref{eq:pUnur} 
give the logarithmic potentials of $\nu_2$ and $\nu_r$
associated with $\delta_p$.
Associated with a general probability measure $\nu = \nu_1$ on
$[0,\infty)$, we then have measures $\nu_2$ and $\nu_r$
whose logarithmic potentials are obtained from 
averaging \eqref{eq:pUnu2} and \eqref{eq:pUnur}
over $p$, that is
\begin{equation} \label{eq:Unu2} 
U^{\nu_2}(z) = \int \log \left| \frac{z^{1/r} - x^{1/r}}{z-x}
\right| d\nu(x)  - \frac{r+t}{t} \log \left| z^{1/r} + q^{1/r} \right|\end{equation}
and
\begin{equation} \label{eq:Unur} 
U^{\nu_r}(z) = \frac{r+t}{r} \log \left| \frac{z^{1/r} + q^{1/r}}{z-(-1)^r q}
\right|  - \int \log \left| z^{1/r} - x^{1/r} \right| d\nu(x).
\end{equation}
From \eqref{eq:Unu2} we obtain
\begin{equation} \label{eq:Inu12}
	I(\nu, \nu_2) = I(\nu) - I_{1/r}(\nu)
		- \frac{r+t}{t} \int \log |x^{1/r} + q^{1/r}| d\nu(x),
\end{equation}
and from \eqref{eq:Unur} we obtain, noting that $z^{1/r}$
in \eqref{eq:Unur} is negative for $z \in \mathbb R \setminus \Delta_r$,
and thus in particular for $z = (-1)^r q$,
\begin{align} \nonumber 
	U^{\nu_r}((-1)^r q) &  \nonumber
	= \frac{r+t}{t}  \log
		\lim_{x \to q} \left| \frac{x^{1/r} - q^{1/r}}{x-q} \right| 
		-  \int \log \left| -q^{1/r} - x^{1/r} \right| d\nu(x) \\
	& = \frac{r+t}{t}  \log
		\left( \frac{1}{r} q^{1/r-1} \right)
		- \int \log \left(x^{1/r} + q^{1/r}\right) d \nu(x). 	 
\label{eq:Unurinq}
\end{align}

Using \eqref{eq:Inu12} and \eqref{eq:Unurinq} in
\eqref{eq:Jnu2} we obtain
\begin{multline} \label{eq:Jnu3}
	J(\nu) =  - \frac{1}{2} I(\nu) + \frac{1}{2} I_{1/r}(\nu) 
	+ \frac{r+t}{t} \int \log \left(x^{1/r} + q^{1/r}\right) d\nu(x) \\
- \frac{1}{2} \left(\frac{r+t}{t}\right)^2  \log
\left( \frac{1}{r} q^{1/r-1} \right). 
\end{multline} 

Finally, comparing \eqref{eq:Jnu} with the 
energy functional \eqref{eq:VEPF2} we obtain for a given $\nu$
on $[0,\infty)$ that
\[ \min_{\nu_2, \ldots, \nu_r} 
	\mathcal E(\nu, \nu_2, \ldots, \nu_r)
		= I(\nu) -  I(\nu,\mu_0) + J(\nu) \]
so that in view of \eqref{eq:Jnu3} and noting
that $\mu_0$ is given by \eqref{eq:mu0}
\begin{multline} \label{eq:minEnu} 
	\min_{\nu_2, \ldots, \nu_r}
	\mathcal E(\nu, \nu_2, \ldots, \nu_r)
	 	=  \frac{1}{2} I(\nu) + \frac{1}{2} I_{1/r}(\nu) \\
		- \frac{1-t}{t} \int \log \left(x+q^{-1}\right) d\nu(x)  
		+ \frac{r+t}{t} \int \log \left(x^{1/r} + q^{1/r}
		\right) d\nu(x) \\
		- \frac{1}{2} \left(\frac{r+t}{t}\right)^2  \log
		\left( \frac{1}{r} q^{1/r-1} \right).  
		\end{multline} 
The left-hand side of \eqref{eq:minEnu} as a functional
on probability measures $\nu$ on $[0,\infty)$ attains
its minimum at $\nu= \mu_1$. Since the
last term on the right in \eqref{eq:minEnu} is only a constant,
independent of $\nu$, the proposition follows.
\end{proof}
	
\subsubsection{Minimum of $V$ is in the support}
\label{subsec332}
In the next step we discuss a general fact about the minimizer for a Muttalib-Borodin
type energy functional \eqref{eq:MBEP}
with $\theta > 0$, and where $V :[0,\infty) \to \mathbb R$ is continuous with
\[ \liminf_{x \to \infty} \left( V(x) - (1+\theta) \log x 
\right) \geq - \infty. \]
Under this condition there is a unique probability
measure  $\mu$ on $[0,\infty)$ that minimizes \eqref{eq:MBEP}.

The following is well-known for the case $\theta =1$, but
apparently has not been observed for general $\theta$.
\begin{lemma} \label{lemma312}
	If $x_0 \geq 0$ is such that $V(x_0) = \min\limits_{x \geq 0}
	V(x)$ then $x_0 \in \supp(\mu)$ where $\mu$ is the 
	probability measure that minimizes \eqref{eq:MBEP}.
\end{lemma}
\begin{proof}
	In this proof we use the notation
	\[ U_{\theta}^{\mu}(x) = \int \log \frac{1}{|x^{\theta}-s^{\theta}|} d \mu(s). \]
	
	The minimizer $\mu$ satisfies, for some constant $\ell$,
	\begin{equation}  \label{eq:lem313proof1} 
		U^{\mu}(x) + U_{\theta}^{\mu}(x) + V(x)
		\begin{cases} = \ell, & \text{ on } \supp(\mu), \\
		\geq \ell & \text{ on } [0,\infty).
		\end{cases} \end{equation}
	Now,
	\[ h(x) = 
	U^{\mu}(x) + U_{\theta}^{\mu}(x) 
	= \int \log  \frac{1}{|x-s|} d\mu(s)
		+ \int \log \frac{1}{|x^{\theta} - s^{\theta}|} d\mu(s),
		\]
	extends into the complex plane where we use the principal branch of $x^{\theta}$,
	i.e., with a branch cut along $(-\infty,0]$.
	Then $h$ is harmonic in 
	$\mathbb C \setminus ((-\infty,0] \cup \supp(\mu))$
	and it tends to $-\infty$ as $|x| \to \infty$.
	By the maximum principle for harmonic functions,
	the maximum of $h$ is attained 
	on  $\supp(\mu) \cup (-\infty,0]$ only.
	
	For $x > 0$ and $s > 0$, it is easy to see
	that $|-x-s| > |x-s|$ and $|(-x)^{\theta} - s^{\theta}|
	= |x^{\theta} e^{\pi i \theta} - s^{\theta}| > |x^{\theta}-s^{\theta}|$. 
	Therefore 
	$ U^{\mu}(-x) < U^{\mu}(x)$ and
	$ U^{\mu}_{\theta}(-x) < U^{\mu}_{\theta}(x)$ for $x > 0$
	which means that $h(-x) < h(x)$ for $x > 0$, and
	therefore the maximum of $h$ is  not attained 
	on $(-\infty,0]$.
	
	Thus the maximum of $h$ is attained on $\supp(\mu)$,
	say at $x_1 \in \supp(\mu)$.
	
	If $x_0 \not\in \supp(\mu)$, then $h(x_0) < h(x_1)$,
	and since $V(x_0) \leq V(x_1)$ by the assumption in
	the lemma, we have a strict inequality 	
	\begin{equation} \label{eq:lem313proof2} 
		h(x_0) + V(x_0) < h(x_1) + V(x_1). 
		\end{equation}
	Since $x_1 \in \supp(\mu)$ the right-hand side of \eqref{eq:lem313proof2}
	is equal to $\ell$, by the equality in 
	\eqref{eq:lem313proof1} 
	It follows that
	\[ U^{\mu}(x_0) + U^{\mu}_{\theta}(x_0) + V(x_0)
	< \ell \]
	which contradicts the inequality in
	\eqref{eq:lem313proof1}. Therefore $x_0 \in \supp(\mu)$.		
\end{proof}

\subsubsection{Proof of part (c) of Theorem \ref{theorem23}.}
\label{subsec34}

In view of Proposition \ref{prop311} and 
Lemma \ref{lemma312} it is enough to show that
the external field \eqref{eq:Vx} attains
its minimum at $x=0$ in case $0 < q < 1$.
This is what we do in the next lemma, and then
part (c) follows.

\begin{lemma} \label{lemma313}
	Let $V$ be given by \eqref{eq:Vx} with $0  < q < 1$. 
	Then for every $t \in (0,1)$
	it is true that
	\[ V(0) = \min_{x \geq 0} V(x). \]
\end{lemma}
\begin{proof}
	Note that $\left(x^{1/r} + q^{1/r}\right)^r 
	\geq x + q$ for $x \geq 0$ by the binomial theorem,
	so that by \eqref{eq:Vx}
	\begin{multline} \label{eq:VxminV0} 
	t V(x) - t V(0) \\ \geq 
	- (1-t) \log\left(x+q^{-1}\right) + 
	\frac{r+t}{r} \log\left(x+ q\right)
	- \left(2- t + \frac{t}{r}\right) \log q. 
	\end{multline}
	Since $0 < q < 1$, we have $x+q  \geq q^2x + q$ for $x \geq 0$
	and therefore we can estimate \eqref{eq:VxminV0}
	further to obtain
	\begin{multline*}  
	t V(x) - t V(0) \\ \geq 
	- (1-t) \log\left(x+q^{-1}\right) + 
	\frac{r+t}{r} \log\left(q^2x+ q\right)
	- \left(2- t + \frac{t}{r}\right) \log q.
	\end{multline*}
	This means
	$V(x) -  V(0) \geq \left(1 + \frac{1}{r}\right) \log\left(qx+1 \right)$  
	and the lemma follows.
\end{proof}

\subsection{Proof of part (d)}
\label{subsec34}

The measure $\mu_1$ minimizes
$I(\mu) - I(\mu,\mu_0+\mu_2)$ among all probability
measures on $[0,\infty)$. Thus it is the equilibrium measure in the external field $-U^{\mu_0+\mu_2}$
which is real analytic on $[0,\infty)$. Then
it follows from \cite{DeKrMc98} that $\mu_1$
is absolutely  continuous  with respect to Lebesgue
measure with a density that is real analytic 
in the interior of its support.

From \cite{DeKrMc98} it also follows that
the density at the endpoint $x_1$ (in $\BIS$ and $\TIS$
cases) behaves as $\approx  c (x_1-x)^{\frac{1}{2} +2N}$
as $x \to x_1-$, for a certain non-negative integer $N$.  
From the iterated balayage it can be seen that $N=0$.
Indeed, the algorithm in section \ref{subsec314}
gives us the sequence of signed measures $(\nu_k)_k$ 
with densities that are such that
$\sqrt{x} \frac{d\nu_k}{dx}$ is positive and 
strictly decreasing
on $[0,x_1]$ by Lemma \ref{lemma34}~(b).
The proof of that lemma (see \eqref{eq:lem32proof3})) actually shows that
the decrease gets stronger as $k$ increases.
Since $\nu_k \to \mu_1$ as $k \to \infty$,
it then follows that the density of $\mu_1$
cannot have a zero derivative at $x_1$,
and thus it vanishes as a square root at $x_1$ 
in $\BIS$ and $\TIS$ cases.

Similarly, the density vanishes as a square root
at $x_2$ in $\UIS$ and $\TIS$ cases.

\section{Proof of Theorem \ref{theorem28}} \label{sec4}

\subsection{Properties of $\Phi$} \label{subsec41}

We start by listing a number of properties of
the meromorphic function $\Phi$ from Definition \ref{def26}. 

\subsubsection{Zeros and poles} \label{subsec411}

\begin{lemma} \label{lemma41} 
	Let $q > 0$.
	The function $z \Phi$ is a degree $2$ meromorphic function
	on the Riemann surface $\mathcal R$ with
	the following properties.
	\begin{enumerate}
	\item[\rm (a)] $z \Phi \to 1$ as $z \to \infty$
		on any of the sheets. 
	\item[\rm (b)] It has simple poles at $z=-q^{-1}$ on the first sheet
	and at $z= (-1)^r q$ on the last sheet, and no
	other poles.
	\item[\rm (c)] Suppose one of the $\BIS$, $\TIS$, 
	or $\FIS$ cases, so that $z \in \supp(\mu_1)$. Then $z \Phi$ has simple zeros at $z=0$ and
	at a point $x_0$ on the first sheet with 
	\begin{equation} \label{eq:x0interval}
	-q^{-1} < x_0 < 0, \end{equation}
	and no other zeros.  In $\UIS$ case, there are two
	points on the Riemann surface with $z=0$
	and  $z \Phi$ has a simple zero at each of these.
	\end{enumerate}	
	
\end{lemma}
\begin{proof}
	We already noted that $\Phi$ is meromorphic on
	$\mathcal R$, see the discussion after Definition \ref{def26}. Thus also $z \Phi$ is meromorphic 	on $\mathcal R$. In part (b) we
	show that it has two simple poles, and no other poles,
	and therefore its degree is two.
	
\medskip	
	(a) 
	From \eqref{eq:Fj} we have $zF_j(z)  \to  \mu_j(\Delta_j)$
	as $z \to \infty$. In view of the total massses \eqref{eq:massmuj} of the measures and the definition
	\eqref{eq:Phij}, part (a) follows.
	
\medskip
	(b) From \eqref{eq:F0} and \eqref{eq:Phij} we see that
	\begin{align} \nonumber 
	\Phi^{(1)}(z) & = t F_1(z) + \frac{1-t}{z+q^{-1}}
		\\ & = t \int \frac{d\mu_1(s)}{z-s} 
		+ \frac{1-t}{z+q^{-1}}, 
		 \label{eq:residue1} 
	 \\ \label{eq:residue2} 
	\Phi^{(r+1)}(z) & = -t F_r(z) + \frac{r+t}{z-(-1)^{r} q}, \end{align}
	and so $z \Phi$ has simple poles at $-q^{-1}$ on the
	first sheet and at $(-1)^r q$ on the $r+1$-st sheet.
	
	There is no pole at $z=\infty$ because of part (a).
	There is no pole at $z=0$ either, since
	the form  \eqref{eq:Fj} of $F_j$ as  a Stieltjes transform, easily implies that $z F_j(z) \to 0$
	as $z \to 0$. Thus also $z \Phi \to 0$ as $z \to 0$.
	There are no other candidates for poles, and
	therefore the degree is two.
	
\medskip
    (c) We already remarked in part (b) that $z \Phi$    vanishes when $z=0$.
    In $\UIS$ case there are two points on the Riemann
    surface with $z=0$. In that case both of these are
    simple zeros, and there are no other zeros, 
    since the degree is two.
    
    In $\BIS$, $\TIS$, and $\FIS$ cases, there is only one
    point $z=0$, and it is at most a double zero of $z \Phi$. 
	Then $z^{\frac{1}{r+1}}$ is the local coordinate,
	and $z$ (as a function on the Riemann surface) has
	a zero of order $r+1$ at $z=0$. Hence
	$\Phi$ has a pole at $z=0$ of order $\geq r-1$,
	and so $\Phi$ is unbounded at $z=0$ (we may assume $r \geq 2$). 
	Looking on the first sheet, we conclude from \eqref{eq:residue1} that $F_1$ is unbounded at
	$z = 0$, and it dominates the behavior of
	$\Phi^{(1)}$ as $z \to 0$.
	Since $F_1(z) < 0$ for negative real $z$, it
	then follows that $\Phi^{(1)}(z)$ is negative for 
	negative $z$ close to $0$.
	
	From \eqref{eq:residue1} we also see that
	$\Phi^{(1)}(z) \to + \infty$ as $z \to -q^{-1}+$,
	as the residue at the pole is positive.
	Thus $\Phi^{(1)}$ changes sign on
	 the interval $(-q^{-1}, 0)$ and hence there is
	 a zero, say at $x_0 \in (-q^{-1},0)$.
	Then $x_0$ is also a zero of $z \Phi$,
	and we conclude that both $z=x_0$ and $z=0$ 
	are simple zeros, and these are the only zeros,
	as the degree of $z \Phi$ is two.
\end{proof}

\subsubsection{Proof of Theorem \ref{theorem23} (e)}
\label{subsec412}

\begin{proof}
Suppose we are in one of the $\BIS$, $\TIS$, or 
$\FIS$ cases, so that $0 \in \supp(\mu_1)$.
Then $z=0$ is a simple zero of $z \Phi$,
by part (c) of Lemma \ref{lemma41}. Since $z^{\frac{1}{r+1}}$
is a local coordinate, we find that $\Phi$ has a pole
of order $r$ at $z=0$. From  \eqref{eq:residue1}
we then get for some non-zero constant $C$,
\[ F_1(z) = C z^{-\frac{r}{r+1}}\left(1+ O\left(z^{\frac{1}{r+1}}\right)\right) \quad 
\text{ as } z \to 0 \]
and the fractional powers have their branch cut along
$[0,\infty)$. 
By the Stieltjes inversion formula
\[ \frac{d\mu_1}{dx} = -\frac{1}{\pi} 
	 \lim_{\delta \to 0+} \Im F_{1}(x+ i \delta),
	 \qquad x > 0,  
	\]
and \eqref{eq:mu1at0} follows.  	

Finally, \eqref{eq:mu1atinf} follows from \eqref{eq:mu1at0}
and the symmetry between $q$ and $1/q$, see Remark \ref{remark24}. \end{proof}

\begin{remark} \label{remark42}
	The behavior \eqref{eq:mu1at0} is characteristic
	for the density of minimizers of Muttalib-Borodin
	type energy functionals as in \eqref{eq:MBEP}.
	This was proved by Claeys and Romano \cite[Remark 1.9]{ClRo14} under general conditions on the
	external field, which however
	do not cover the case \eqref{eq:Vx}. 
\end{remark}

\subsubsection{Critical points}
We need to know about the critical points, by which
we mean the ramification points of $z \Phi$. There
is no ramification at $z=0$ or $z=\infty$, and so
we may alternatively characterize the critical points
as those points where the derivative of 
$z \Phi^{(j)}(z)$ vanishes for some $j=1, \ldots, r+1$.

Since $z \Phi$ has degree $2$ the Riemann-Hurwitz formula
\cite{Sc14}
tells us that there are two critical points 
in the $\BIS$ and $\FIS$ (genus zero) cases, and four
critical points in the $\TIS$ (genus one) case. 
The following lemma says that they are all real and on the first sheet.

We fix $0 < q < 1$ and we continue to use $x_1, x_2$ 
as in Theorem \ref{theorem23}~(a)
depending on the  various cases, and $x_0$ for the zero
of $\Phi$ on the first sheet as in Lemma~\ref{lemma41}~(c).

\begin{figure}[t]
	\centering
	\begin{overpic}[scale=0.4]{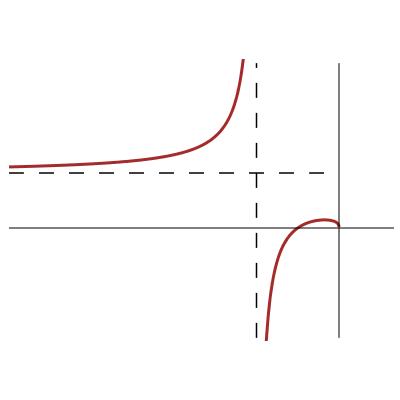}
		\put(54,35){$-q^{-1}$}
		\put(69,50){$x_0$}
		\put(76,35){$y_1$}
		\put(82,54){$1$}
		\put(72,43){$\vdots$}
		\put(79,39){$\vdots$}
		\put(0,75){$z\Phi^{(1)}$ in $\BIS$ case}
	\end{overpic}
	\begin{overpic}[scale=0.4]{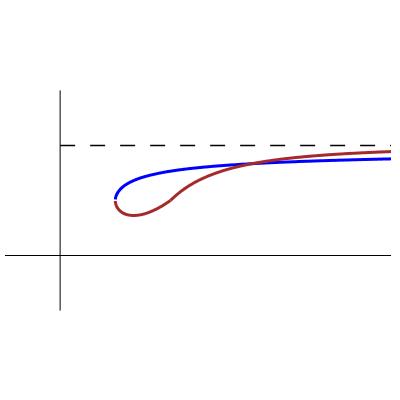}
		\put(25,30){$x_1$}
		\put(32,30){$y_2$}
		\put(10,60){$1$}
		\put(27,36){$\vdots$}
		\put(27,43){$\vdots$}
		\put(33,36){$\vdots$}
		\put(33,39){$\vdots$}
		\put(18,75){$z\Phi^{(1)}$ and $z\Phi^{(2)}$ in $\BIS$ case}
	\end{overpic}
	\caption{Sketches of the graph of $z\Phi$ in $\BIS$ case. The left panel shows the graph of $z \Phi^{(1)}(z)$ on the negative real line where it has a pole at $-q^{-1}$,
	a zero at $x_0$ and a local maximum at $y_1$.
	The right panel shows the graphs of $z \Phi^{(1)}(z)$
	(in brown)
	and $z \Phi^{(2)}$ (in blue) on the interval $[x_1,\infty)$. 
	The graph of $z \Phi^{(1)}(z)$ has a local minimum
	at $y_2$, while the graph of $z \Phi^{(2)}(z)$
	is strictly increasing. Both graphs tend to $1$
	at infinity. 
	The graph of $z\Phi^{(r+1)}(z)$ is as in the right panel of Figure \ref{figure5} below. \label{figure3}}
\end{figure}
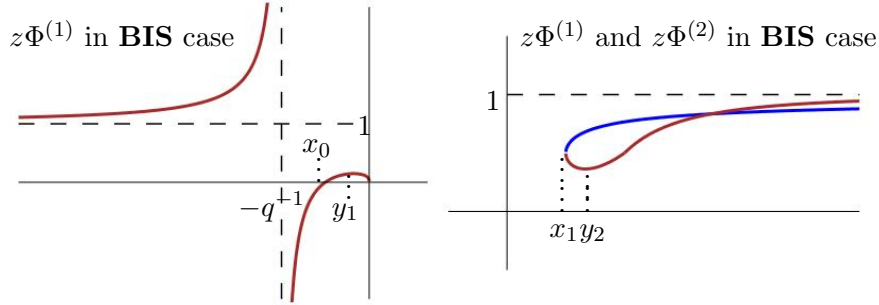

\begin{lemma} \label{lemma43}  Let $0 < q < 1$.
	The critical points of $z \Phi$ are on the real part of
	the first sheet of the Riemann surface. 
	\begin{enumerate}
	\item[\rm (a)] In all cases
	there is a critical point $y_1$ with $y_1 \in (x_0,0)$.
	\item[\rm (b)]  In $\BIS$ case there is one
	more critical point $y_2 \in (x_1, \infty)$. 
	\item[\rm (c)]  In $\TIS$ case there are three more
	critical points. A critical point $y_0 \in (-\infty, -q^{-1})$ and two critical points $y_2, y_3
	\in (x_1,x_2)$ with $y_2 < y_3$. 
	\item[\rm (d)] In $\FIS$ case there is one
	more critical point $y_0 \in (-\infty, -q^{-1})$.
\end{enumerate}
\end{lemma}

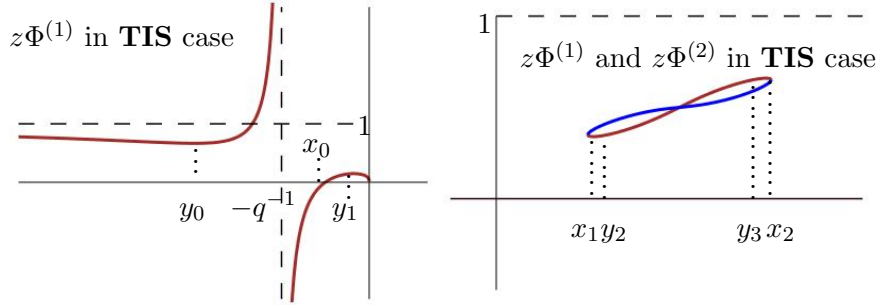
\begin{figure}[t]
	\centering
	
	\begin{overpic}[scale=0.4]{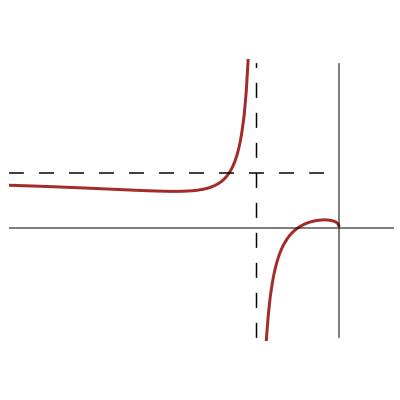}
		\put(40,35){$y_0$}
		\put(52,35){$-q^{-1}$}
		\put(69,50){$x_0$}
		\put(76,35){$y_1$}
		\put(82,54){$1$}
		\put(43,45){$\vdots$}
		\put(72,43){$\vdots$}
		\put(79,39){$\vdots$}
		\put(0,75){$z\Phi^{(1)}$ in $\TIS$ case}
	\end{overpic}
	\begin{overpic}[scale=0.4]{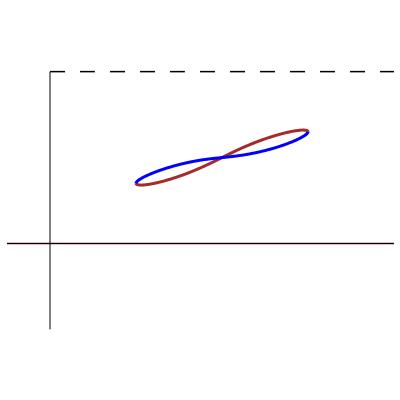}
		\put(30,30){$x_1$}
		\put(37,30){$y_2$}
		\put(8,78){$1$}
		\put(34,40){$\vdots$}
		\put(34,47){$\vdots$}
		\put(37,39){$\vdots$}
		\put(37,46){$\vdots$}
		
		\put(76,30){$x_2$}
		\put(69,30){$y_3$}
		\put(76,40){$\vdots$}
		\put(76,47){$\vdots$}
		\put(76,54){$\vdots$}
		\put(76,61){$\vdots$}
		\put(72,39){$\vdots$}
		\put(72,46){$\vdots$}
		\put(72,53){$\vdots$}
		\put(72,60){$\vdots$}
		\put(18,70){$z\Phi^{(1)}$ and $z\Phi^{(2)}$ in $\TIS$ case}
	\end{overpic}
	\caption{Sketches of the graph of $z\Phi$ in $\TIS$ case. The left panel shows the graph of $z \Phi^{(1)}(z)$ on the negative real line where it has a pole at $-q^{-1}$,
	zeros at $x_0$ and $0$, a  local minimum at $y_0$, and a local maximum at $y_1$.
		The right panel shows the graphs of $z \Phi^{(1)}(z)$
		(in brown)
		and $z \Phi^{(2)}$ (in blue) on the interval $[x_1,x_2]$. 
		The graph of $z \Phi^{(1)}(z)$ has a local minimum
		at $y_2$ and a local maximum at $y_3$, while the graph of $z \Phi^{(2)}(z)$
		is strictly increasing. 
		The graph of $z\Phi^{(r+1)}(z)$ is as in the right panel of Figure \ref{figure5} below. \label{figure4}}
\end{figure}

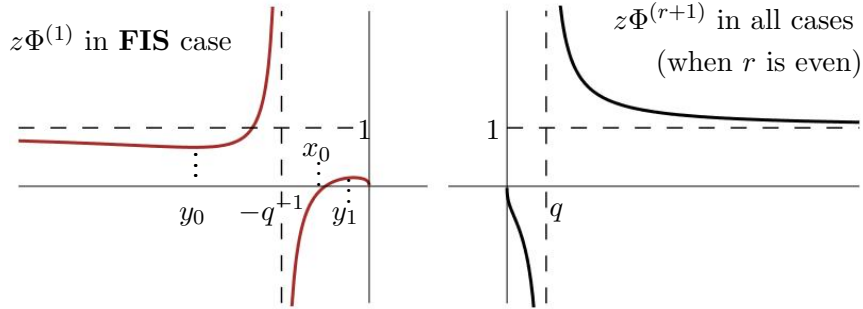
\begin{figure}[t]
	\centering	
	\begin{overpic}[scale=0.4]{TISgraph1}
		\put(40,35){$y_0$}
		\put(54,35){$-q^{-1}$}
		\put(69,50){$x_0$}
		\put(76,35){$y_1$}
		\put(82,54){$1$}
		\put(43,45){$\vdots$}
		\put(72,43){$\vdots$}
		\put(79,39){$\vdots$}
		\put(43,45){$\vdots$}
		\put(0,75){$z\Phi^{(1)}$ in $\FIS$ case}
	\end{overpic}
	\begin{overpic}[scale=0.4]{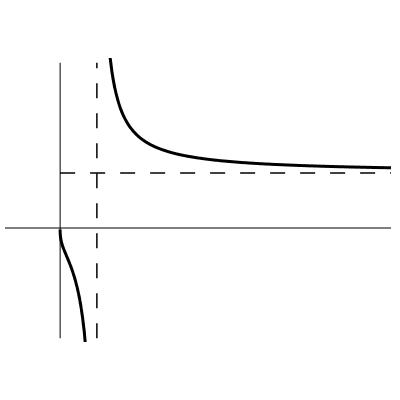}
	\put(25,36){$q$}
	\put(10,54){$1$}
	\put(40,80){$z\Phi^{(r+1)}$ in all cases}
	\put(50,70){(when $r$ is even)}
	\end{overpic}
	\caption{Sketches of the graph of $z\Phi$ in $\FIS$ case. The left panel shows the graph of $z \Phi^{(1)}(z)$ on 
	$(-\infty,0]$ where it has a pole at $-q^{-1}$,
	 zeros at $x_0$ and $0$, a  local minimum at $y_0$, and a local maximum at $y_1$.
	The right panel shows the graph of $z \Phi^{(r+1)}(z)$ on $(-1)^r [0,\infty)$ which has a pole at $(-1)^r q$. The figure is for $r=2$ and it has the
	same features for all cases.
	\label{figure5}}
\end{figure}

\begin{proof}
	(a) By Lemma \ref{lemma41} (c) 
	$z \Phi^{(1)}(z)$ has zeros at $z=x_0$ and at $z=0$, and
	in between it is real and positive. So there is a local
	maximum, and the point $y_1 \in (x_0, 0)$ where it is attained is a  critical point in all cases.
	
	\medskip
	
	The proofs of parts (b)-(d) rely on an inspection 
	of the graph of $z\Phi$ on the  real part of the
	Riemann surface (that is, on the part where both
	$z$ and $\Phi$ are real), see Figures \ref{figure3},
	\ref{figure4}, and \ref{figure5} for sketches
	of the graphs in the various cases.  
	We infer the following about $z\Phi$ from the 
	behavior at the poles and at infinity,
	\begin{itemize} 
	\item  every value in $(-\infty,0)$ is attained 
	once in $(-q^{-1}, x_0)$ on the first sheet 
	and once between $0$  and $(-1)^r q$ on the last sheet,
	\item every value in $(1,\infty)$ is attained
	once in $(-\infty, -q^{-1})$ on the first sheet
	and once between $(-1)^r q$ and infinity
	on the last sheet. 
	\end{itemize}
	Since $z\Phi$ has degree two, the values
	in $(-\infty,0)$ and $(1,\infty)$ are attained nowhere
	else on the Riemann surface.
	In particular
	\begin{equation} \label{eq:zPhiin01}
	\begin{array}{l}
	0 < z \Phi^{(1)}(z) < 1, \text{ and } \\[5pt]
	0 <  z \Phi^{(2)}(z) < 1, \end{array} \quad
	\begin{cases} \text{ for $z \in [x_1,\infty)$} & 
	\text{in $\BIS$ case}, \\
	\text{ for $z \in [x_1,x_2]$} & 
	\text{in $\TIS$ case},
	\end{cases}
	\end{equation} 
	see the right panels of Figures \ref{figure3} and \ref{figure4}.
	
	From the fact that the density of $\mu_1$ vanishes 
	as a square root at the endpoints $x_1, x_2$ 
	see Theorem \ref{theorem23} (d), it follows
	that
	\begin{equation} \label{eq:F1prime} 
		F_1'(z) = - \int \frac{1}{(z-x)^2} d\mu_1(x)
		\to - \infty	
	\end{equation}
	as  $z \to x_1+$ or $z \to x_2-$.
	Hence by \eqref{eq:Phij} we also have
	\begin{equation} \label{eq:zPhi1prime}
	\left(z \Phi^{(1)}(z)\right)' \to - \infty \quad \text{ and } 
	\quad 	\left(z \Phi^{(2)}(z)\right)' \to + \infty 
	\end{equation}
	as $z \to x_1+$ (in $\BIS$ and $\TIS$ cases) 
	or $z \to x_2-$ (in $\TIS$ case). 
	
	\medskip
	(b) In $\BIS$ case we noted in \eqref{eq:zPhiin01} 
	that $z \Phi^{(1)}(z)$ 
	takes the value $x_1 \Phi^{(1)}(x_1) \in (0,1)$  
	at $x_1$, and by \eqref{eq:zPhi1prime} it 
	starts to decrease if
	$z \in (x_1,\infty)$ increases. Since it tends to 
	the value $1$ at infinity,
	there will be a local minimum, say at $y_2 \in (x_1,\infty)$. This is a critical point, and
	part (b) follows.  
	
	It also follows that $z \Phi^{(2)}(z)$ strictly increases
	for $z \in [x_1,\infty)$ since in $\BIS$ case there
	are no further critical points.
	
	\medskip
	
	(c) In $\TIS$ case we first observe that
	\begin{equation}
		\frac{d}{dz} \left( z F_1(z) \right)	
			= - \int \frac{x}{(z-x)^2} d\mu_1(x) < 0,
				\qquad z \in (x_1,x_2), 
	\end{equation}
	\begin{equation}
	\frac{d}{dz} \left( z F_2(z) \right)	
	= - \int \frac{x}{(z-x)^2} d\mu_2(x) > 0,
	\qquad z \in (x_1,x_2). 
	\end{equation}	
	The difference in sign is due to the fact
	that $\mu_1$ is supported on $[0,\infty)$,
	while $\mu_2$ is supported on $(-\infty,0]$.
	Thus by \eqref{eq:Phij}
	\[ \left(z \Phi^{(2)}(z)\right)'
		 > 0 \quad \text{ for } z \in (x_1,x_2), \]
	and therefore $z \Phi^{(2)}(z)$ strictly increases on
	$(x_1,x_2)$ and there are no critical points in 
	$(x_1,x_2)$ on the second sheet. We also conclude
	\[ x_1 \Phi^{(1)}(x_1) <  x_2 \Phi^{(1)}(x_2) \]
	but $z \Phi^{(1)}(z)$ will not be monotonic on
	$[x_1,x_2]$ due to \eqref{eq:zPhi1prime}. Instead
	it will start to decrease at $x_1$ to a local
	minimum, say at $y_2$, and then increases to a local
	maximum, say at $y_3$, and then again decreases.
	This gives us the critical points $y_2 < y_3$ in
	$(x_1, x_2)$. We already know $y_1 \in (-x_0, 0)$.
	
	The final critical point is in $(-\infty, -q^{-1})$,
	and this follows from the observation that
	\[ z F_1(z) = z \int \frac{d\mu_1(x)}{z-x}
		= 1 + \int \frac{x d\mu_1(x)}{z-x} \]
	with $\int \frac{x d\mu_1(x)}{z-x} < 0$ for $z \in (-\infty,0]$ and 
	\begin{equation} \label{eq:F1atinf} 
	\int \frac{x d\mu_1(x)}{z-x} =
	 -C z^{- \frac{1}{r+1}}( 1 + O(z^{-\frac{1}{r+1}}) 
	 	\text{ as } z \to -\infty, \end{equation}
	with a positive constant $C > 0$.
	Also 
	$z F_0(z) = (1-t^{-1}) + O(z^{-1})$ as $z \to \infty$,
	so that by 	\eqref{eq:Phij}.
	\begin{equation} \label{eq:zPhi1atinf} 
	z \Phi^{(1)}(z) = 1 + 
	\int \frac{x d\mu_1(x)}{z-x}  + O(z^{-1}) 
	\end{equation}
	as $z \to -\infty$, where the second term is
	negative for $z < 0$ and it dominates the $O(z^{-1})$
	term as $z \to -\infty$. Therefore 
	$z \Phi^{(1)}(z)$ decreases on an interval
	$(-\infty, y_0)$ for some $y_0 \in (-\infty, -q^{-1})$, 
	it reaches a local minimum	at $y_0$ and then increases
	to $+\infty$ as $z \to -q^{-1}-$. See also Figure \ref{figure4}.
	
	\medskip
			
	(d) In $\FIS$ case the expansions \eqref{eq:F1atinf}
	and \eqref{eq:zPhi1atinf} remain valid, as does the
	conclusion that $z \Phi^{(1)}(z)$ has a local
	minimum at some $y_0 \in (-\infty, -q^{-1})$, 
	and $y_0$ is a critical point.  
\end{proof}

After these preparations we  turn to the proof
of Theorem \ref{theorem28}.

\subsection{Proof of part (a)} \label{subsec42}

\begin{proof}
For $x \in \supp(\mu_1)$, we have by \eqref{eq:Phij}
and the Stieltjes inversion formula
\begin{align}  x \Im \Phi^{(1)}_{\pm}(x) 
&	=  xt \Im \left(F_1\right)_{\pm}(x)  = \mp \frac{xt}{\pi} \frac{d\mu_1(x)}{dx}.
\label{eq:xPhi1onsupp} 
\end{align}
Here the subscript $\pm$ denotes the limiting value from the upper
($+$) or lower ($-$) half plane.
Then by \eqref{eq:defU} we find 
$\supp(\mu_1) \subset U$.
 
from the Cauchy-Riemann equations and the definition
\eqref{eq:defU} of $U$, we obtain that the parts
of the real line where $z \Phi^{(1)}(z)$ is
real and decreasing belong to $U$, while those
parts where $z \Phi^{(1)}(z)$ is real and increasing
do not belong to $U$. Then in view of the behavior
of $z \Phi^{(1)}(z)$ on the realline that we see in Figures \ref{figure3},
\ref{figure4}, \ref{figure5}, and the fact
that $\supp(\mu_1) \subset U$, we conclude that
\begin{equation}
U  \cap (\mathbb R \cup \{\infty\}) 
= \begin{cases}
[y_1,y_2] & \text{in $\BIS$ case}, \\
[-\infty,y_0] \cup [y_1,y_2] \cup [y_3,\infty]
& \text{in $\TIS$ case}, \\
[-\infty,y_0] \cup [y_1,\infty] &
\text{in $\FIS$ case}.
\end{cases}
\end{equation}
In particular $-q^{-1} \not\in U$. This proves part (a).
\end{proof}

\subsection{Proof of part (b)} \label{subsec43}
\begin{proof}
Since the $y_j$'s are critical points, we have 
that $z \Phi^{(1)}(z)$ is also
real on certain contours that emanate from 
each $y_j$ into the complex plane.
These contours are going to be the boundary $\partial U$
of $U$.

The labelling of the critical points in Lemma \ref{lemma43}
is such that $y_j$ is a local minimum of $z \Phi^{(1)}(z)$ 
if $j$ is even,
and a local maximum if $j$ is odd, when we restrict
to the real line. It means that  $z \Phi^{(1)}(z)$
is real and increasing on $\partial U$
when we move away from $y_j$ with $j$ odd,
and decreasing from $y_j$ with $j$ even. 

Noting that  
\begin{align} \label{eq:critvalBIS} 
 y_1 \Phi^{(1)}(y_1) <  y_2 \Phi^{(1)}(y_2) & \quad
\text{in $\BIS$ case}, 
\end{align} 
we conclude that the part of $\partial U$ that
emanates from $y_1$ will end at $y_2$ in $\BIS$
case.  Since $\Sigma_1 =
\supp(\mu_1) \subset U$, we also see that 
$\partial U$ consists of a simple closed contour
surrounding $\Sigma_1$
and $U$ is a bounded simply connected domain
in the $\BIS$ case.
\end{proof}

\subsection{Proof of part (c)} \label{subsec44}

\begin{proof}
In $\TIS$ case, we have four critical points
and instead of \eqref{eq:critvalBIS} we have
\begin{align}
\label{eq:critvalTIS} 
 y_1 \Phi^{(1)}(y_1) <  y_2 \Phi^{(1)}(y_2)
< y_3 \Phi^{(1)}(y_3) < y_0 \Phi^{(1)}(y_0) & \quad
\text{in $\TIS$ case}. 
\end{align}
Then $\partial U$ consists
of two closed contours, one containing $y_1$ and $y_2$,
and one containing $y_0$ and $y_3$. Both closed 
contours go around $\Sigma_1$. It follows that 
$U$ has two  components, namely the  bounded domain
that is enclosed by the inner component of $\partial U$,
and the  unbounded domain that is outside of
the outer component of $\partial U$.
This proves part (c). \end{proof}

\subsection{Proof of part (d)} \label{subsec45}

\begin{proof}
In $\FIS$ case we have two critical points $y_0 < y_1 < 0$
with
\begin{align}
\label{eq:critvalFIS} 
y_1 \Phi^{(1)}(y_1) <  y_0 \Phi^{(1)}(y_0), & \quad
\text{in $\FIS$ case}.
\end{align}
Then $\partial U$  is a closed contour 
containing $y_0$ and $y_1$,
and $U$ is the domain that is exterior to this contour.
\end{proof}

\subsection{Proof of part (e)} \label{subsec43}

\begin{proof}
We know from Theorem \ref{theorem23} (b) 
that $t \mapsto t \mu_{j,t}$ increases with $t$
for every $j$, and $t \mu_{j,t}$ has total mass
$t+ j-1$. Then
\begin{equation} \label{eq:rhoj} 
	\rho_j = \rho_{j,t} = \frac{\partial (t\mu_{j,t})}{\partial t}, \quad  j=0, \ldots, r+1
\end{equation} 
is a probability measure on
$\supp(\mu_j)$ for every $j$. In particular $\rho_0 = \delta_{-q^{-1}}$ by \eqref{eq:massmuj}. 

Thus by differentiating \eqref{eq:Phij} for $j=1$
with respect to $t$,
\begin{align*} \frac{\partial \Phi^{(1)}(z)}{\partial t}  
	& =  \int \frac{d\rho_1(x)}{z-x}  - \frac{1}{z+q^{-1}}.
\end{align*}
Since $\rho_0$ and $\rho_1$ are both probability measures
we obtain from this
	\begin{align} \label{eq:zPhi1partialt} 
	\frac{\partial (z \Phi^{(1)}(z))}{\partial t}	& = \int \frac{x d\rho_1(x)}{z-x} + \frac{q^{-1}}{z+q^{-1}}.
	\end{align}

For $\Im z > 0$ both terms in the right-hand side of  \eqref{eq:zPhi1partialt} 
	have negative imaginary parts (since
	$x d\rho_1(x)$ is a positive measure), 
	while for $\Im z < 0$
	the two terms have positive imaginary parts.
	In other words
	\[ \frac{\partial}{\partial t} 
	\Im \left(z \Phi^{(1)}(z)\right) 
	\begin{cases} < 0 \qquad \text{ for } \Im z >  0, \\
	> 0 \qquad \text{ for } \Im z < 0, 
	\end{cases} \]
	Therefore the part in the upper half plane where
	$\Im \left(z \Phi^{(1)}(z)\right) < 0$ increases with $t$. Similarly,
	the part in the lower half plane where
	$\Im \left(z \Phi^{(1)}(z)\right) > 0$ increass with $t$,
	which proves part (e) in view of the definition
	\eqref{eq:defU} of $U$.
	\end{proof}

\subsection{Proof of part (f)} \label{subsec44}

\begin{proof}
It is clear from \eqref{eq:defU} that 
$z \Phi^{(1)}(z)$ is real-valued for $z \in \partial U$.
The point of part (f) is that $\partial U$
is characterized by \eqref{eq:partialU}.

Being a meromorphic function on a compact Riemann surface,
$\Phi$ satisfies the algebraic equation
\begin{equation} \label{eq:thm24fproof1} 
 \prod_{j=1}^{r+1} \left(\Phi - \Phi^{(j)}(z) \right) =	\Phi^{r+1}  + \sum_{k=1}^{r+1} (-1)^k e_k(z) \Phi^{r+1-k}
= 0 \end{equation}	
where $e_k(z)$ is the $k$th elementary symmetric function
in $\Phi^{(1)}, \ldots, \Phi^{(r+1)}$, i.e.,
\begin{equation}  \label{eq:thm24fproof2}
	e_k(z) = \sum_{1 \leq j_1 < \cdots < j_k \leq r+1} \prod_{l=1}^k \Phi^{(j_l)}(z). 
\end{equation}
Each $e_k$ is a rational function of $z \in \mathbb C$ 
with real coefficients and simple poles at $-q^{-1}$ and at $(-1)^rq$,
due to the simple poles of $\Phi^{(1)}$ and $\Phi^{(r+1)}$
at these respective values, see \eqref{eq:residue1} and
\eqref{eq:residue2}. 
Since $z \Phi^{(j)}(z) \to 1$ as $z \to \infty$
for every $j$ by Lemma \ref{lemma41} (a), and there are $\binom{r+1}{k}$ terms
in \eqref{eq:thm24fproof2} we have
\[ z^k e_k(z) \to \binom{r+1}{k} \quad \text{ as } z \to \infty. \]
By Lemma \ref{lemma41} (c) $z=0$ is a simple zero 
of $z \Phi$ on the Riemann surface, and from  \eqref{eq:thm24fproof2} we get that 
$z^k e_k(z)$ becomes zero for $z=0$.
Thus	
\[ z^k e_k(z) =  
	\binom{r+1}{k}  \frac{z(z+A_k)}{(z+q^{-1})(z-(-1)^r q)}  \]
for some real value $A_k$.

Using this in \eqref{eq:thm24fproof1} and clearing denominators by multiplying with $z^r(z+q^{-1})(z-(-1)^rq)$ we obtain
\begin{multline} \label{eq:thm24fproof3} 
		z^r \left(z+q^{-1}\right)\left(z-(-1)^r q\right) \Phi^{r+1} \\
	+\sum_{k=1}^{r+1} (-1)^k \binom{r+1}{k} (z+A_k) 
	\left(z \Phi \right)^{r+1-k} = 0.
	\end{multline}
We separate terms that are polynomial in $z \Phi$
to  rewrite \eqref{eq:thm24fproof3} as 	(with $A_0 = q^{-1}-(-1)^r q$)
\begin{multline} \label{eq:thm24fproof4}
	\sum_{k=0}^{r+1} (-1)^k \binom{r+1}{k} A_k (z \Phi)^{r+1-k} \\
	 = (-1)^r z^{-1} \left(z \Phi\right)^{r+1}
	- z \sum_{k=0}^{r+1} (-1)^k \binom{r+1}{k}
	\left(z \Phi\right)^{r+1-k} \\ 
	= (-1)^r z^{-1} \left(z \Phi\right)^{r+1} + z \left(z\Phi-1\right)^{r+1}. \end{multline}
In the last step we used the binomial theorem.
	
Let $z \in \partial U$. 	 
Then $z \Phi^{(1)}(z)$ is real and it satisfies
the equation \eqref{eq:thm24fproof4}, which means that the left-hand side is real since each $A_k$ is real.
Thus the right-hand side is real as well, and 
taking  imaginary parts we obtain since $z \Phi^{(1)}(z)$ is real,
	\[ 0 = (-1)^{r+1} \frac{\Im z}{|z|^2} \left(z \Phi^{(1)}(z) \right)^{r+1}
		+ \Im z \left(z \Phi^{(1)}(z) - 1 \right)^{r+1},
		\quad \text{ for } z \in \partial U. \]
This equation leads to \eqref{eq:partialU} whenever $\Im z \neq 0$.
Thus \eqref{eq:partialU} holds for
$z \in \partial U \setminus \mathbb R$ and by
continuity it also holds for $z \in \partial U
\cap  \mathbb R$. \end{proof}

\section{Proof of Theorem \ref{theorem211}} \label{sec5}

\subsection{Proof of part (a)} \label{subsec51}

\begin{proof}
Since $\mu^*$ is the symmetric pullback of the probability
measure $\mu_1$, it is also a probability measure.
That $\mu_{\Omega}$ is a probability measure as
well can be seen from the formulas  in Lemma \ref{lemma52}
below, by letting $z \to \infty$ in either
\eqref{eq:lemma52formula1} (in case $\Omega$ is bounded),
or \eqref{eq:lemma52formula2} (in case $\Omega$ is unbounded). 
\end{proof}

\subsection{Proof of part (b)} \label{subsec52}

\begin{proof}
Since $t \mu_{1,t}$ is increasing by Theorem \ref{theorem23} (b), also $t \mu^*_t$ increases with $t$.

The domains $U_t$ increase with $t$ by Theorem \ref{theorem28}.
Then also $\Omega_t$ increases with $t$ 
and then $t \mu_{\Omega,t}$ also increases with
$t$, since by \eqref{eq:muOmega} this is just the
spherical area measure $\frac{1}{\pi} \frac{dA(z)}{(1+|z|^2)^2}$ restricted to $\Omega_t$.  
\end{proof}

\subsection{Stieltjes transform of $\mu_{\Omega}$} \label{subsec53}

The proofs of parts (c) and (d) are modelled after the proofs
in the paper \cite{CrKu19+} that deals with the case $r=1$. 
See in particular the proof of Proposition 4.1 
in \cite{CrKu19+}. 
As a preparation we need the following formula for the spherical
Schwarz function from \eqref{eq:defS}, which is the analogue of \cite[(5.1)]{CrKu19+}.

\begin{lemma} \label{lemma51}
We have
\begin{align} S(z) & = \frac{(1-t) z^r}{z^{r+1} + q^{-1}}
+ t \int \frac{d\mu^*(x)}{z-x} \label{eq:Szformula} 
\end{align}
\end{lemma}
\begin{proof}
Since $\mu^*$ is the symmetric pullback of $\mu_1$
under the mapping $z \mapsto z^{r+1}$, we can easily
verify that their logarithmic potentials are related
via 
\[ U^{\mu^*}(z) = \frac{1}{r+1} U^{\mu_1}(z^{r+1}). \]
and also, with appropriate branches of the logarithm,
\[ \int \log(z-x) d\mu^*(x) =
	\frac{1}{r+1} \int \log\left(z^{r+1}-x\right) d\mu_1(x). \]
Taking the $z$-derivative we find
\begin{equation} \label{eq:lemma51proof1} 
	\int \frac{d\mu^*(x)}{z-x} =  z^r F_1(z^{r+1}). 
	\end{equation}
Then combining \eqref{eq:defS}, \eqref{eq:residue1}, and
\eqref{eq:lemma51proof1}, we find  \eqref{eq:Szformula}.
\end{proof}
	
The following lemma is the analogue of \cite[Lemma 5.2]{CrKu19+}
and its proof is also very similar.
\begin{lemma} \label{lemma52}
	The Stieltjes transform of $\mu_{\Omega}$ satisfies
	\begin{align} \label{eq:lemma52formula1}
	\int \frac{d\mu_\Omega(x)}{z-x}
		& = \int \frac{d\mu^*(x)}{z-x},
			\qquad z \in \mathbb C \setminus \Omega, \\
  	\int \frac{d\mu_\Omega(x)}{z-x}	&  = -\frac{(1-t)z^r}{t(z^{r+1} + q^{-1})} 
			+ \frac{\bar{z}}{t(1+|z|^2)},
			\quad z \in \Omega. \label{eq:lemma52formula2}
	\end{align}	
\end{lemma}
\begin{proof} 
Take $z \in \mathbb C \setminus \Omega$ first.
Then by \eqref{eq:dmuOmega}, \eqref{eq:muOmega}, the complex Green's formula,
and the property \eqref{eq:sphSchwarz} of the
spherical Schwarz function, we find if $\Omega$ is bounded,
\begin{align} \nonumber
	t \int \frac{d\mu_{\Omega}(s)}{z-s} & =
	\frac{1}{\pi} \int_{\Omega} \frac{dA(s)}{(z-s) (1+|s|^2)^2} \\
	& = \frac{1}{2\pi i} \oint_{\partial \Omega}
		\frac{\bar{s}}{(z-s) (1+|s|^2)} ds \nonumber \\
	& = \frac{1}{2\pi i} \oint_{\partial \Omega}
		\frac{S(s)}{z-s} ds \nonumber \\
	& = - \frac{1}{2\pi i} \oint_{\partial (\mathbb C \setminus \Omega)}
	\frac{S(s)}{z-s} ds.  \label{eq:lemma52proof1} 
		\end{align}
Since the complex Green's formula applies to bounded
domains, one has to modify the calculation in case
$\Omega$ is unbounded. Then one first makes a cut-off to $\{ z \in \Omega \mid |z| \leq R\}$ with  a large $R > 0$.
The Green's formula then produces an additional integral 
over $|z|=R$, which however
tends to zero as $R \to \infty$, due to the fact that $\frac{S(s)}{z-s} = O(s^{-2})$ as $s \to \infty$. Thus \eqref{eq:lemma52proof1}
also holds in the unbounded case.

The remaining integral in \eqref{eq:lemma52proof1} 
is evaluated with the residue
theorem for $\mathbb C \setminus \Omega$.
The spherical Schwarz function $S$ has  $r+1$ simple  poles at the solutions of $s^{r+1} = -q^{-1}$, the
poles are all in $\mathbb C \setminus \Omega$, and from \eqref{eq:Szformula}
it can be checked that $S$ has the same 
residue $\frac{1-t}{r+1}$ at each of the poles. Together 
they give the contribution
\begin{equation} \label{eq:lemma52proof2} 
	-\frac{1-t}{r+1} \sum_{s^{r+1} = - q^{-1}}
	\frac{1}{z-s} = - \frac{(1-t)z^r}{z^{r+1}+q^{-1}}
	\end{equation}
to the integral \eqref{eq:lemma52proof1}.
There is an additional pole in \eqref{eq:lemma52proof1}
at $s = z$ with the contribution $S(z)$. Finally, note
that there is no contribution from infinity in case 
$\mathbb C \setminus \Omega$ is unbounded, since the integrand in
\eqref{eq:lemma52proof1} is $O(s^{-2})$ as $s \to \infty$.
 In total we get 
\[ t \int_{\Omega} \frac{d\mu_{\Omega}(s)}{z-s}
	= -  \frac{(1-t) z^r}{z^{r+1}+q^{-1}} + S(z),
	\qquad z \in \mathbb C \setminus \Omega,
 \]
and \eqref{eq:lemma52formula2} follows because of \eqref{eq:Szformula}.

\medskip
Let $z \in \Omega \setminus \partial \Omega$. Take 
$\varepsilon >0$ such that the disk $D(z,\varepsilon)$
of radius $\varepsilon$ around $z$ is contained in
$\Omega$.
Then by a calculation similar to \eqref{eq:lemma52proof1}, 
with complex Green's theorem and the spherical Schwarz function
\begin{align} \nonumber 
t  \int_{\Omega \setminus D(z,\varepsilon)}
	\frac{d\mu_{\Omega}(s)}{z-s}
		& = -\frac{1}{2\pi i} \oint_{\partial (\mathbb C\setminus \Omega)} 	
			\frac{S(s)}{z-s} ds \\ 
		& \qquad  - \frac{1}{2\pi i} \oint_{\partial D(z,\varepsilon)}
	\frac{ \bar{s}}{(z-s) (1+|s|^2)} ds. 
	\label{eq:lemma52proof3} \end{align}
The integral over $\partial (\mathbb C \setminus  \Omega)$ is again evaluated using the residue theorem, but in the present
situation there is no contribution from $s=z$, but only the combined contribution
\eqref{eq:lemma52proof2} from the poles of $S$. 
The  integral over the circle $\partial D(z,\varepsilon)$ (including the prefactor $- \frac{1}{2\pi i}$) tends
to $\frac{\bar{z}}{1+|z|^2}$ as $\varepsilon \to 0+$.
Thus letting $\varepsilon \to 0+$ in \eqref{eq:lemma52proof3} we obtain
\eqref{eq:lemma52formula2}.
\end{proof}

\subsection{The log integral of $\mu_{\Omega}$} \label{subsec54}

We will need the following result in the $\TIS$ case 
for the proof of Lemma \ref{lemma56} below.

\begin{lemma} \label{lemma53}
	In $\FIS$ and $\TIS$ cases we have
	\begin{equation} \label{eq:Umuat0}
	U^{\mu_{\Omega}}(0) + \frac{1-t}{(r+1)t} \log q
	= 0. \end{equation}
\end{lemma}
In $\BIS$ case we have 
\[ U^{\mu_{\Omega}}(0) + \frac{1-t}{(r+1)t} \log q
	= \frac{1}{2t(r+1)} \int_{x_1}^{\infty} 
	\left(\Phi^{(1)}(x) - \Phi^{(2)}(x) \right) dx \leq 0 
\]
but we will not prove this  as we do not need it for the
proof of parts (c) and (d) of Theorem \ref{theorem211}.

\begin{proof}[Proof of Lemma \ref{lemma53}.]
In the proof we use $\log x = \log |x| + i \arg x$ 
with $0 < \arg x < 2\pi$, and we are
going to show that 
	
\begin{equation} \label{eq:logmuOmega} 
	t\int \log(x) d\mu_{\Omega}(x) = 
	\frac{1-t}{r+1} \log q + t \pi i \quad  
	\text{ in $\FIS$ and $\TIS$ cases}
\end{equation}
and then \eqref{eq:Umuat0} will follow by taking
the real parts on both sides.
The evaluation of \eqref{eq:logmuOmega} follows 
along the lines of the proof of Lemma \ref{lemma52}
but there is a non-trivial extra step 
required in the $\TIS$ case. 

\medskip
We start with the $\FIS$ case. 
Consider the cut-off domain 
\[ \Omega_{R,\delta} =
	\{ z \in \Omega \mid |z| \leq R,
		\dist(z, [0,\infty)) > \delta \}		\] 
with large $R > 0$ and small $\delta > 0$.
Due to our definition of the logarithm with the 
branch cut along $[0,\infty)$ we can apply the
complex Green's theorem to the integral over
$\Omega_{R,\delta}$ and we find in $\FIS$ case
\begin{multline} \label{eq:lemma53proof1}
	\frac{1}{\pi} \int_{\Omega_{R,\delta}} 
	\log x \frac{dA(x)}{(1+|x|^2)^2}
	 = \frac{1}{2\pi i} \oint_{\partial \Omega_{R,\delta}}
	 \log s \frac{\overline{s}}{1+|s|^2} ds \\
	 \to \frac{1}{2\pi i} \oint_{\Omega} 
	 	\log s \frac{\overline{s}}{1+|s|^2} ds 
		+ \frac{1}{2\pi i} \oint_{|s| = R} 
		 \log s \frac{\overline{s}}{1+|s|^2} ds \\
	 		- \int_0^R \frac{x}{1+x^2} dx 
	\end{multline}	
as $\delta \to 0+$. The last term in \eqref{eq:lemma53proof1} is the combined 
contribution of the upper and lower sides of
the branch cut of the logarithm. 
It yields
\begin{equation} \label{eq:lemma53proof2} 
\int_{0}^R \frac{x}{1+x^2} dx = \frac{1}{2}
\log \left(1+R^2\right) = \log R + o(1) 
\quad \text{ as } R  \to \infty. \end{equation}

The second integral in the right-hand side of \eqref{eq:lemma53proof1} is evaluated with  parametrization $s = R e^{i \theta}$, $0 < \theta < 2 \pi$, to give
\begin{align} \nonumber
	\frac{1}{2\pi i} \oint_{|s| = R}
	\log(s) \frac{\overline{s}}{1+|s|^2} ds
	 &	= \frac{R^2 \log R}{1+R^2}
			+ \frac{R^2}{1+R^2} \pi i \\
	\label{eq:lemma53proof3}
	&	= \log R + \pi i  + o(1) \quad \text{ as } R \to \infty. \end{align}
In the first term we use \eqref{eq:sphSchwarz}
and then evaluate the integral by a residue calculation
over $\mathbb C \setminus \Omega$. The complement
of $\Omega$ consists of $r+1$ disjoint disks 
in $\FIS$ case, see Figure \ref{figureOmega}, 
and $S(s)$ is meromorphic with one simple pole
at the solution of $s^{r+1} + q^{-1}$ in each of the disks
with residue $\frac{1-t}{r+1}$. Therefore 
\begin{align}  \nonumber
	\frac{1}{2\pi i} \oint_{\Omega} \log s
	\frac{\overline{s}}{1+|s|^2} ds & = 
	-\frac{1}{2\pi i} \oint_{\partial (C \setminus \Omega)}
	(\log s) S(s) ds \\ \nonumber
	& = -\frac{1-t}{r+1} 
		\sum_{s: s^{r+1} = -q^{-1}} \log s \\
	\label{eq:lemma53proof4} & = 
		\frac{1-t}{r+1} \log q - (1-t) \pi i   
		\end{align}
		
Letting $R \to \infty$ in \eqref{eq:lemma53proof1} 
we find from \eqref{eq:muOmega}, \eqref{eq:lemma53proof2}, \eqref{eq:lemma53proof3} and \eqref{eq:lemma53proof4} that
\begin{align*} t \int \log(x) d\mu_{\Omega}(x)
	& = \lim_{R \to \infty} \lim_{\delta \to 0+}
		\frac{1}{\pi} \int_{\Omega_{R,\delta}} 	
		\log x \frac{dA(x)}{(1+|x|^2)^2} \\
	& = \frac{1-t}{r+1} \log q + t \pi i \end{align*}
as claimed in \eqref{eq:logmuOmega} in $\FIS$ case.

\medskip

In $\TIS$ case we have to adjust the above calculation in
two ways. First, since $\Omega \cap [0,\infty) = [0,y_2^*]
\cup [y_3^*,\infty)$, with $y_j^* = y_j^{1/(r+1)}$, the integral over $[0,R]$ in \eqref{eq:lemma53proof1} is replaced by 
\[ \left(\int_0^{y_2^*} + \int_{y_3^*}^R \right) \frac{x}{1+x^2} dx = - \int_{y_2^*}^{y_3^*} \frac{x}{1+x^2} dx 
	+ \log R + o(1) \quad \text{ as } R \to \infty.
	\]
Second, in the evaluation \eqref{eq:lemma53proof4} of
the integral over $\Omega$, there is a contribution 
from the intersection $[y_2^*,y_3^*]$ 
of $\mathbb C \setminus \Omega$ with the positive real line, due to the discontinuity of the logarithm.
Instead of \eqref{eq:lemma53proof4} we get
\[ \frac{1}{2\pi i} \oint_{\Omega}  \log s 
	\frac{\overline{s}}{1+|s|^2} ds
  = \frac{1-t}{r+1} \log q - (1-t) \pi i	
  	-  \int_{y_2^*}^{y_3^*} S(x) dx, \]
and the result is the formula
\[ t \int \log(x) d\mu_{\Omega}(x) = 
\frac{1-t}{r+1} \log q + t\pi i
	- \int_{y_2^*}^{y_3^*} \left(S(x) - \frac{x}{1+x^2}
	\right) dx \]
for the $\TIS$ case.
 
To obtain \eqref{eq:logmuOmega} it remains to prove that
the integral in the right-hand side vanishes,
and we do this by showing the identity 
\eqref{eq:lem54TIS} in Lemma \ref{lemma54} below.
The right-hand side of \eqref{eq:lem54TIS} is 
zero and to see this we recall  \eqref{eq:Phij} 
from which we get
\begin{multline*}
		\Phi^{(1)}(x) -  \Phi^{(2)}(x)  = 
		t(2 F_1(x) - F_0(x) - F_2(x)) \\
	    = -t \frac{d}{dx} \left(2 U^{\mu_1}(x) -  
		U^{\mu_0}(x) - U^{\mu_2}(x) \right),
		\quad \text{for } x_1 < x < x_2. 
		\end{multline*}
We also recall that  $2 U^{\mu_1} - U^{\mu_0} - U^{\mu_2}$ 
vanishes on the support of $\mu_1$, and hence
in particular at both $x_1$ and $x_2$.
Then the right-hand side is indeed $0$ by the
fundamental theorem of calculus.

Lemma \ref{lemma53} is thus proved, pending the proof of
the remarkable identity 
\eqref{eq:lem54TIS}. Since the proof of this identity
uses new ideas that were not in \cite{CrKu19+},
we decided to give it in a separate lemma.
\end{proof}

\begin{lemma} \label{lemma54}
	In the $\TIS$ case we have
\begin{align} 
\int_{y_2^*}^{y_3^*} \left(S(x) - \frac{x}{1+x^2} \right) dx 
= \frac{1}{2(r+1)} \int_{x_1}^{x_2}
\left( \Phi^{(1)}(x) - \Phi^{(2)}(x) \right) dx, 
\label{eq:lem54TIS} 
\end{align}
where $y_j^* = y_j^{\frac{1}{r+1}}$ for $j=2,3$.
\end{lemma}
\begin{proof}
Consider $\omega = \Phi dz$ as a meromorphic differential on the Riemann surface. Then 
 \begin{equation} \label{eq:lem54proof2} 
	 	\oint_{a} \omega = 
	 	\int_{x_1}^{x_2} \left(\Phi^{(1)}(x) - \Phi^{(2)}(x) \right) dx \end{equation}
	 for the cycle $a$ that goes from $x_1$ to
 	$x_2$ on the first sheet, and back from $x_2$ to $x_1$ on the second sheet, cf.~Figure \ref{figure2}. The meromorphic differential
 	has simple poles at $-q^{-1}$ on first sheet,
 	at $(-1)^r q$ on last sheet, and at $\infty$
 	with respective residues $1-t$, $r+t$, and $-1-r$.
 	
 	Since $z \Phi$ is a degree two meromorphic function we can represent the Riemann
 	surface $\mathcal R$ by the equations
 	\begin{equation} \label{eq:lem54proof3} 
 		\eta^2 =  \prod_{j=0}^3 (\zeta - \zeta_j),
 		\qquad \zeta = z \Phi \end{equation}
 	with $\zeta_j = 
 	y_j \Phi^{(1)}(y_j)$, for $j=0,1,2,3$, 
 	being the four branch points
 	of $\zeta$ with $y_0 < y_1 < y_2 < y_3$ by Lemma \ref{lemma43} and 
 	\[ 0 < \zeta_1 < \zeta_2 < \zeta_3 < \zeta_0 < 1, \]
 	see also Figure \ref{figure4}. In the new coordinates $\mathcal R$ is a two sheeted cover of the $\zeta$-plane, with
 	branch cuts $[\zeta_1, \zeta_2]$ and $[\zeta_3,\zeta_0]$. We label the sheets so
 	that $z= -q^{-1}$ corresponds to 
 	$\zeta = \infty$ on the first sheet and $z= (-1)^r q$ to $\zeta = \infty$ on the second sheet.
 	Then $\eta$ is positive for real 
 	$\zeta > \zeta_0$ on the first sheet.
 	The point $z=\infty$ corresponds to the point 
 	$\zeta = 1$ on the second sheet.
 	The
 	$a$-cycle goes from $\zeta_2$ to $\zeta_3$ on
 	the first sheet and back from $\zeta_3$ to $\zeta_2$
 	on the second sheet.
 	
 	The meromorphic differential $\omega$ has 
 	simple poles at the two points at $\zeta= \infty$ 
 	with residues $1-t$ and $r+t$, and at
 	$\zeta = 1$ on the second sheet
 	with residue $-1-r$.  	Then 
 	\begin{equation} \label{eq:lem54proof4}
 	 2\omega + (r+1) \frac{d\zeta}{\zeta-1} \end{equation}
 	has residues $\pm (1-2t -r)$  	at the two
 	points at infinity, and residues 
 	$\pm (r+1)$ at the two points with $\zeta = 1$. Thus 
 	\eqref{eq:lem54proof4} has an 
 	anti-symmetry with respect to the 
 	involution $(\zeta, \eta) \mapsto
 	(\zeta, -\eta)$ of $\mathcal R$.  It follows that
 	\[ 2 \omega + (r+1) \frac{d\zeta}{\zeta-1}
 		= \frac{A \zeta^2 + B \zeta + C}{\zeta-1} 
 		\frac{d\zeta}{\eta} \] 
 	for certain constants $A$, $B$, and $C$. From this
 	form we conclude that 
 	\[ \oint_{a} 
 	\left(2 \omega + (r+1) \frac{d\zeta}{\zeta-1} \right) d\zeta
 	= 2 \int_{\zeta_2}^{\zeta_3} 
 		\left( 2 \omega + (r+1) \frac{d\zeta}{\zeta-1} \right)
 		d\zeta \]
 	with integration on the first sheet.
 	Clearly
 	$\oint_a \frac{d\zeta}{\zeta-1} = 0$ and therefore
	\begin{align} \nonumber
	\oint_a \omega  & =   
	\int_{\zeta_2}^{\zeta_3} \left( 2 \omega 
	+ (r+1) \frac{d\zeta}{\zeta-1}  \right) \\
	\label{eq:lem54proof5}
	& = 2 \int_{y_2}^{y_3} \Phi^{(1)}(z) dz +
	(r+1) \left( \log(1-\zeta_3) - \log(1-\zeta_2) \right)
\end{align}
	since $[\zeta_2,\zeta_3]$
 	on the first sheet corresponds to $[y_2,y_3]$ on the first sheet in the original $z$-variable where 
 	$\omega = \Phi^{(1)}(z) dz$.
 	
 	Changing variable $z = x^{r+1}$ and using \eqref{eq:defS} we have
 	\begin{align} \label{eq:lem54proof6}
 	\int_{y_2}^{y_3} \Phi^{(1)}(z) dz = (r+1) \int_{y_2^*}^{y_3^*} S(x) dx
 	\end{align}
 	since $y_j^* = y_j^{\frac{1}{r+1}}$. For the last
 	term on the right of \eqref{eq:lem54proof5}
 	we recall that $\zeta_j = y_j \Phi^{(1)}(y_j)$ and $y_j$ belongs to $\partial U$.
    Therefore  it satisfies the equation \eqref{eq:partialU}, that is,
    \[ \zeta_j = \frac{y_j^{\frac{2}{r+1}}}{1 +  y_j^{\frac{2}{r+1}}} = \frac{(y_j^*)^2}{1+ (y_j^*)^2}, \qquad \text{for } j = 2,3, \]
    which we rewrite as   
    \begin{align} \label{eq:lem54proof7}
        \log\left(1- \zeta_j\right) =	
     - \log\left(1+  (y_j^*)^2\right),
     \qquad \text{for } j =2,3.
     \end{align}
 	Hence 
 	\begin{align} \nonumber
 		\log \left(1-\zeta_2\right) - \log\left(1-\zeta_3\right)
 		& = \log\left(1+(y_3^*)^2\right) - \log\left(1+(y_2^*)^2\right) \\
 		& = 2 \int_{y_2^*}^{y_3^*} \frac{x}{1+x^2} dx.
 		\label{eq:lem54proof8}
 		\end{align}
 	Combining \eqref{eq:lem54proof2},  \eqref{eq:lem54proof5}, \eqref{eq:lem54proof6}, \eqref{eq:lem54proof8}
 	we obtain the  equality of the two
 	integrals in \eqref{eq:lem54TIS}.
\end{proof}

\subsection{Measures $\nu_t$ and $\rho_t$} \label{subsec55}

For the proofs of  parts (c) and (d), we also need to consider
the dynamical picture where we vary $t$, see
\cite[section 6]{CrKu19+}. To emphasize
the $t$-dependence we attach a subscript $t$ to the
notions that vary with $t$.

We already observed in part (b) that $t\mu^*_t$ and $t \mu_{\Omega,t}$ increase with $t$. The derivatives 
\begin{equation} \label{eq:rhotnut} 
	\rho_{t} = \frac{\partial}{\partial t} \left(t\mu^*_t\right),
	\qquad
	\nu_{t} = \frac{\partial}{\partial t} 
	\left(t \mu_{\Omega,t}\right)
	\end{equation}
therefore exist for almost every $t$,
as can be proved as in \cite[Theorem 2]{BuRa99},
but in our case the derivatives actually exist for every $t \in (0,1)$.

Both $\rho_t$ and $\nu_t$ are probability measures, with
$\supp(\rho_t) = \supp(\mu^*_t)$ and
$\supp(\nu_t)  = \partial \Omega_t$,
see \eqref{eq:muOmega}. Indeed $\nu_t$ measures 
how the domain $\Omega_t$ grows in the spherical metric
as $t$ increases. 

Applying $\frac{\partial}{\partial t} t$ to the identities \eqref{eq:lemma52formula1}
and \eqref{eq:lemma52formula2} for the Stieltjes 
transforms, and using \eqref{eq:rhotnut}, we get
\begin{align} \label{eq:Fnu}
	 \int \frac{d\nu_t(s)}{z-s} 
	  = \begin{cases} \ds \int \frac{d\rho_t(s)}{z-s}, 
	  & \qquad z \in \mathbb C \setminus \Omega_t, \\[10pt] 
	\ds \frac{z^r}{z^{r+1}+q^{-1}}, & \qquad z \in \Omega_t.
	\end{cases} 
	\end{align}
	
\begin{lemma} \label{lemma56}
	There are constants $C_{1,t}$ and $C_{2,t}$
	such that the following hold.
	\begin{enumerate}
	\item[\rm (a)] We have
		\begin{equation} \label{eq:Unutineq1} 
		U^{\nu_t}(z) \leq U^{\rho_t}(z) + C_{2,t},
		\qquad z \in \mathbb C,
		\end{equation}
		with equality for $z \in \mathbb C \setminus \Omega_t$.
		\item[\rm (b)]  We have
		\begin{align} \label{eq:Unutineq2} 
		U^{\nu_t}(z) \leq - \frac{1}{r+1} \log |z^{r+1} + q^{-1}| + C_{1,t}, 
		\quad  z \in \mathbb C,
		\end{align} 
	with equality for $z  \in \Omega_t$.
		
\end{enumerate}
	\end{lemma}

\begin{proof}
(a) 
The first identity in \eqref{eq:Fnu} implies that $U^{\nu_t} - U^{\rho_t}$
is constant on each connected component of $\mathbb C \setminus \Omega_t$. Thus for some constant $C_{2,t}$,
\begin{align} \label{eq:Unuteq1}
	U^{\nu_t}(z) =   U^{\rho_{t}}(z) + C_{2,t}, 
	\quad  z \in \mathbb C \setminus \Omega_t,
\end{align}
since $\mathbb C \setminus \Omega_t$ is either connected
(in $\BIS$ and $\TIS$ cases), or consists of $r+1$
disjoint components (in $\FIS$ case) where  due to $r+1$-fold rotational symmetry  the constant is the
same on each component. If $\Omega_t$ is bounded (the $\BIS$ case) then $C_{2,t} =  0$, since both
potentials in \eqref{eq:Unuteq1} behave as
$-\log|z| + o(1)$ as $z \to \infty$. 

Since $\nu_t$ is supported on $\partial \Omega_t$, the function $U^{\rho_t} - U^{\nu_t}$
is superharmonic on the interior of $\Omega_t$,
including at $\infty$ if $\Omega_t$ is unbounded.
By the minimum principle for superharmonic functions we find 
the  corresponding inequality \eqref{eq:Unutineq1} on $\Omega_t$, and part (a) follows.

\medskip

(b) 
For part (b) we argue similarly, but there is an additional
twist when $\Omega_t$ is not connected (the $\TIS$ case).
Using the second identity  of \eqref{eq:Fnu},
we apply similar reasoning to $U^{\nu_t}$
and $- \frac{1}{r+1} \log|z^{r+1} + q^{-1}|$, 
which is the logarithmic potential of the discrete
measure with mass $\frac{1}{r+1}$ at each solution
of $z^{r+1} + q^{-1} = 0$.
We find that 
\[ U^{\nu_t}(z) + \frac{1}{r+1} \log |z^{r+1} + q^{-1}| \]
is constant on each connected component of $\Omega_t$.

In $\BIS$ and $\FIS$ cases we have that $\Omega_t$ is
connected and therefore for some constant
$C_{1,t}$,
\begin{align} \label{eq:Unuteq2}
	U^{\nu_t}(z) & =  - \frac{1}{r+1} \log |z^{r+1} + q^{-1}| + C_{1,t},  \qquad  z \in \Omega_t,	\end{align}
in $\BIS$ and $\FIS$ cases. If $\Omega_t$ is unbounded
then we let $z \to \infty$ in \eqref{eq:Unuteq2}
and we find that $C_{1,t} = 0$ in $\FIS$ case.

In $\TIS$ case we have that $\Omega_t$ has two
connected components. We find that \eqref{eq:Unuteq2}
holds with $C_{1,t} =0$ in the unbounded component
for the same reason that $C_{1,t} = 0$ in $\FIS$ case.
The bounded component could potentially have
a different constant. However we are able to
compute $U^{\nu_t}(z)$ at $z=0$ because
of Lemma \ref{lemma53} which says that
\[ t U^{\mu_{\Omega,t}}(0) = - \frac{1-t}{r+1} \log q \]
in the $\TIS$ case. Then taking the $t$-derivative
and using the definition \eqref{eq:rhotnut} of $\nu_t$,
we obtain
\[ U^{\nu_t}(0) = \frac{1}{r+1} \log q, \]
which implies that \eqref{eq:Unuteq2} with $C_{1,t}=0$
holds for $z=0$ and thus throughout the 
bounded component as well in the $\TIS$ case.

From \eqref{eq:Unuteq2} and the fact that $\nu_t$
is supported on $\partial \Omega_t$, we obtain
the inequality \eqref{eq:Unutineq2} in all cases
(by the minimum principle, as in the proof of part (a))
and part (b) follows.
\end{proof}

\begin{remark}
The identity \eqref{eq:Unuteq1} and the fact that
$\supp(\nu_t) = \partial \Omega_t$ show in fact that
\[ \nu_t = \Bal(\rho_t, \partial \Omega_t). \]
Similarly \eqref{eq:Unuteq2} gives that  
\[ \nu_t = \Bal\left( \frac{1}{r+1} \sum_{z^{r+1} = - q^{-1}} 
\delta_z, \partial \Omega_t\right). \]
Thus $\nu_t$ is a balayage measure onto $\partial \Omega_t$ from two
sides. It is the balayage of $\rho_t$
which is supported inside $\Omega_t$,
and it is  also the balayage of a discrete measure 
supported in the complement on $\Omega_t$.  
\end{remark}

\subsection{Proof of part (c)} \label{subsec56}

\begin{proof}
Integrating the identities \eqref{eq:rhotnut} we obtain
the identities
\begin{align} \label{eq:BuRa}  
	t \mu^*_t = \int_0^t \rho_{s} ds, \quad
	\text{and} \quad
	 t \mu_{\Omega,t} = \int_0^t \nu_s ds, \end{align}
which  are analogous to the formulas of Buyarov and
Rakhmanov \cite{BuRa99} for varying families of
measures on the real line.
We also have
\begin{align} \nonumber
	t \mu_{\Omega,t} & = 
	\lim_{\tau \to 1-} \tau \mu_{\Omega,\tau}
	- \int_t^1 \nu_s ds \\  \label{eq:BuRa2} 
	& = \frac{dA(z)}{\pi (1+|z|^2)^2} 
	- \int_t^1 \nu_s ds.
	\end{align}
We can calculate the logarithmic potential 
\[ - \int_{\mathbb C} \log|z-s|
	 \frac{dA(s)}{\pi (1+|s|^2)} = - \frac{1}{2} \log\left(1+|z|^2\right), \qquad z \in \mathbb C. \]
Therefore by \eqref{eq:BuRa2}
and \eqref{eq:Unutineq2}, we have 
for every $z \in \mathbb C$, 
\begin{align} \nonumber
	 t U^{\mu_{\Omega,t}}(z) & = - \frac{1}{2}
	 \log \left(1+|z|^2\right) - 
	 \int_t^1 U^{\nu_s}(z) ds \\ \nonumber
	 & \geq - \frac{1}{2} \log \left(1+|z|^2\right)	  
	 + \int_t^1 \left(\frac{1}{r+1} \log \left|z^{r+1}+q^{-1}\right| - C_{1,s} \right) ds \\
	 & = - \frac{1}{2} \log \left(1+|z|^2\right)	
	 + \frac{1-t}{r+1} \log\left|z^{r+1}+q^{-1}\right|  + t c_{1,t},  \label{eq:UmuOmegaineq1}
	 \end{align}
with $c_{1,t} = - \ds \frac{1}{t} \int_t^1 C_{1,s} ds$.	 

Equality  holds in \eqref{eq:Unutineq2} for $z \in \Omega_t$ which implies that 
equality holds in \eqref{eq:UmuOmegaineq1}
for
\[  z \in \bigcap_{t \leq s < 1} \Omega_s = \Omega_t, \]
since $\Omega_t \subset \Omega_s$ whenever $t < s$.
The proof of part (c) is complete.  
\end{proof}

\subsection{Proof of part (d)}

\begin{proof}
We obtain for every $z \in \mathbb C$, using 
\eqref{eq:BuRa} and \eqref{eq:Unutineq1},
\begin{align*} 
	t U^{\mu_{\Omega,t}}(z)
	& = \int_0^t U^{\nu_s}(z) ds \\
	& \leq \int_0^t \left( U^{\rho_s}(z) + C_{2,s} \right) ds \\
	& = t U^{\mu^*_t}(z) + t c_{2,t}
	\quad \text{ with } 
	c_{2,t} = \frac{1}{t} \int_0^t C_{2,s} ds.
	\end{align*}
Equality holds, by Lemma \ref{lemma56} (a),  for 
\[ z \in \bigcap_{0 < s \leq t} 
	\left( \mathbb C \setminus \Omega_s  \right) 
	= \mathbb C \setminus \Omega_t, \]
since $\Omega_s \subset \Omega_t$ whenever $s < t$.
\end{proof}

\end{document}